%% file: reedy.tex
\pdfoutput=1

\documentclass[12pt,reqno]{amsart}

\usepackage{a4wide}
\usepackage{amsmath}
\usepackage{amssymb}
\usepackage{amsxtra}
\usepackage{amsthm}
\usepackage{mathtools}
\usepackage{extpfeil}
\usepackage{enumitem}
\usepackage[mathscr]{eucal}
\usepackage{graphicx}
\usepackage{ifnextok}
\usepackage[T1]{fontenc}

\usepackage[all]{xy}
\SelectTips{cm}{}
\ifx\pdfoutput\undefined
  \xyoption{dvips}
\fi

\usepackage{hyperref}
\usepackage{slashed}

\input{macros}

\setlist{}
\setenumerate{leftmargin=*,labelindent=\parindent}
\setitemize{leftmargin=*,labelindent=0.5\parindent}
\setdescription{leftmargin=1em}

\swapnumbers

\theoremstyle{plain}
\newtheorem{thm}{Theorem}[section]
\newtheorem{lem}[thm]{Lemma}
\newtheorem{cor}[thm]{Corollary}

\newtheorem{prop}[thm]{Proposition}

\theoremstyle{definition}
\newtheorem{defn}[thm]{Definition}
\newtheorem{ex}[thm]{Example}
\newtheorem{ntn}[thm]{Notation}

\theoremstyle{remark}
\newtheorem{obs}[thm]{Observation}
\newtheorem{rec}[thm]{Recall}

\makeatletter
\let\c@equation\c@thm
\makeatother
\numberwithin{equation}{section}

\title{The theory and practice of Reedy categories}

\author[Riehl]{Emily Riehl}
\address{
  Department of Mathematics \\
  Harvard University \\
  Cambridge, MA 02138\\
  USA
}
\email{eriehl@math.harvard.edu}

\author[Verity]{Dominic Verity}
\address{
  Centre of Australian Category Theory \\
  Macquarie University \\
  NSW 2109 \\
  Australia
}
\email{dominic.verity@mq.edu.au}

\date{$2\nd$ June 2014}


\begin{document}
  
  \ifpdf
  \DeclareGraphicsExtensions{.pdf, .jpg, .tif}
  \else
  \DeclareGraphicsExtensions{.eps, .jpg}
  \fi
  
  \begin{abstract}
      The goal of this paper is to demystify the role played by the Reedy category axioms in homotopy theory. With no assumed prerequisites beyond a healthy appetite for category theoretic arguments, we present streamlined proofs of a number of useful technical results, which are well known to folklore but difficult to find in the literature. While the results presented here are not new, our approach to their proofs is somewhat novel. Specifically, we reduce the much of the hard work involved to simpler computations involving {\em weighted colimits\/} and {\em Leibniz (pushout-product) constructions}. The general theory is developed in parallel with examples, which we use to prove that familiar formulae for homotopy limits and colimits indeed have the desired properties.
  \end{abstract}
  
  \maketitle

\tableofcontents

\section{Introduction}

The homotopical behavior of limits or colimits of diagrams of a fixed shape is somewhat subtle. The situation is improved considerably when the shape in question, given in general by a small category $\scat{C}$, is a {\em Reedy category\/}. Ordinary colimits indexed by Reedy categories include pushouts, coequalisers, coproducts, and sequential colimits. Weighted colimits further include geometric realizations of simplicial objects. As the opposite of a Reedy category is always a Reedy category, the dual limit notions will also fit into this framework.

A Reedy category might admit multiple Reedy category structures, as we shall see in examples below. A key ingredient in a Reedy structure is a degree function assigning to each object a natural number. Certain, but not necessarily all, morphisms from an object to an object of lower (respectively higher) degree are said to {\em strictly lower\/} (respectively {\em raise\/}) {degrees}. Each morphism must factorise uniquely as one of the former type followed by one of the later. 

A consequence of these axioms, providing the key motive for their introduction, is that diagrams indexed by a Reedy category can be defined inductively via a procedure that we will outline below. Trivial examples of this are familiar: a diagram indexed by the poset $\omega$ may of course be stipulated by first choosing the image of the initial object, then choosing a morphism with this object as its domain, then choosing a morphism with the codomain of this chosen arrow as its domain, and so on. In a widely circulated preprint, Chris Reedy described an inductive procedure that can be used to define a simplicial object: the extension from an $n$-truncated simplicial object to an $(n+1)$-truncated simplicial object is determined by a factorisation of a canonical map from the $n$-skeleton to the $n$-coskeleton. This procedure is functorial: extensions of maps also correspond to factorisations, this time in the arrow category.

It will be of interest to understand the homotopical behavior of the skeleta and coskeleta. For instance, the geometric realization of a simplicial space $X$ is the colimit of the geometric realizations of a sequence of maps $\sk_{n-1} X \to \sk_n X$. A geometric understanding of these maps will supply conditions under which this colimit is homotopy invariant. To that end, Reedy observed that there is a pushout diagram of the following form 
\begin{equation}\label{eq:reedy's.po}
  \xymatrix{ \latch_n X \times \Delta^n \cup X_n \times \partial \Delta^n \ar[d] \ar[r] & X_n \times \Delta^n \ar[d] \\ \sk_{n-1} X \ar[r] & \sk_n X \poexcursion}
\end{equation}
in which $\latch_n X$ is the {\em space of degenerate $n$-simplices\/} of $X$, defined in this case to be the union of the images of the degeneracy maps with codomain $X_n$.\footnote{To make this diagram of simplicial spaces ``type check'' we should explain what we mean by the product  $X_n\times \Del^n$ of  a space $X_n$ with a simplicial set $\Del^n$. This is an unfortunate, but traditional, abuse of notation under which if $X$ is a space and $Y$ is a simplicial set then $X\times Y$ denotes the simplicial space whose space of $n$-simplices is the $Y_n$-fold copower of the space $X$.} These pushouts simply formalise our intuition that we may construct the $n$-skeleton of a simplicial space $X$ by adjoining a space of non-degenerate $n$-simplices to its $(n-1)$-skeleton by glueing along suitable maps of simplex boundaries.

The point of view championed in this paper is that Reedy's insights can be reduced to an analysis of the hom-bifunctor $\Delta \colon \Del\op \times \Del \to \Set$. Specifically, we define its $n$-skeleton $\sk_n \Delta$ to be the sub-bifunctor of $\Delta$ of those simplicial operators $[m]\to[k]$ which factorise through $[n]$ and define the boundaries of the covariant and contravariant representables $\Del^n\defeq\Delta(-,[n])$ and $\Del^n\defeq\Delta([n],-)$ to be their sub-functors $\boundary\Del^n\defeq\sk_{n-1}\Delta(-,[n])$ and $\boundary\Del_n\defeq\sk_{n-1}\Delta([n],-)$. Then we show, by a simple and direct combinatorial argument, that there exists a pushout diagram
\begin{equation}\label{eq:reedy's.delta}
  \xymatrix{ \boundary\Delta_n \etimes \Delta^n \cup \Delta_n \etimes \boundary\Delta^n \ar[d] \ar[r] & \Delta_n \etimes \Delta^n \ar[d] \\ \sk_{n-1}\Delta \ar[r] & \poexcursion \sk_n \Delta}
\end{equation}
 in $\Set^{\Del\op \times \Del}$, wherein the symbol $\etimes$ denotes the manifest exterior product functor $\Set^{\Del\op}\times\Set^{\Del}\to\Set^{\Del\op \times \Del}$. Of course this result can be regarded as being a special case of Reedy's, but for us it is a key and very concrete observation which can be extended immediately to his general case of all simplicial objects in arbitrary (small cocomplete) categories. Specifically, we do this by observing that the objects in Reedy's pushout~\eqref{eq:reedy's.po} can be obtained as {\em weighted colimits\/} whose {\em weights\/} are the bifunctors in~\eqref{eq:reedy's.delta}. Everything else follows by observing that weighted colimit constructions are cocontinuous in their weights. This approach generalise with no extra effort to categories of functors on an arbitrary Reedy category, and we present our work here at that level of generality.

In this context, a {\em weight\/} is a functor describing the ``shape'' of a generalised cone over diagrams indexed by a fixed small category. Objects representing the set of cones described by a particular weight are called {\em weighted limits or colimits}. Weighted limits and colimits are indispensable to enriched category theory but can provide a useful conceptual simplification even in the unenriched context, which is all that we will need here. Their use will simplify the proofs involved in the development of Reedy category theory precisely because calculations involving the weights are the reason why these results are true. We summarise our expository project with the slogan that ``it's all in the weights''. 

The definition of a Reedy category originates in unpublished notes of Dan Kan, which circulated as an early draft of the book \cite{Dwyer:2004fk}. This material survives in the published literature thanks to~\cite[chapter 15]{Hirschhorn:2003:ModCat}. A briefer account based on the same source material can be found in~\cite[chapter 4]{Hovey:1999fk}. We offer this as a ``second generation'' account of the classical theory. Other work that might contest this title is \cite{Berger:2008uq} which introduces \emph{generalised Reedy categories}, a definition which, in contrast with ordinary Reedy categories, is invariant under equivalence.

\subsection*{\S\ Roadmap}

We conclude this introduction with a review of the notions of weighted limit and colimit which are central to all of our work here. Reedy categories are introduced in \S\ref{sec:reedy}, in which we prove a lemma characterising the factorisations of arrows in a Reedy category. In \S\ref{sec:latching}, we introduce latching and matching objects, defined using skeleta and coskeleta functors appropriate to the Reedy setting. In \S\ref{sec:Leibniz-Reedy}, we give a thorough study of Leibniz constructions, which are used to define relative latching and matching maps. This section concludes with the definition of the Reedy model structure, which we establish in \S\ref{sec:proof}.

In \S\ref{sec:Leibniz-rel-cell-cx}, we prove that the Leibniz tensor of a pair of relative cell complexes is again a relative cell complex, whose cells are the Leibniz tensors of the given cells. This calculation enables us, in \S\ref{sec:cellular} which is really the heart of this paper, to introduce the canonical cellular presentation for the two-sided representable functor of a Reedy category. As an immediate corollary, any natural transformation between functors indexed by a Reedy category admits a canonical presentation as a relative cell complex whose cells are built from its relative latching maps, and also a canonical presentation as a ``Postnikov tower'' whose layers are built from its relative matching maps (proposition \ref{prop:building-up} and its dual). This makes the proof of the Reedy model structure in \S\ref{sec:proof} essentially a triviality.

As an epilogue, in \S\S\ref{sec:hocolim}--\ref{sec:simplicial} we turn our attention to homotopy limits and colimits of diagrams of Reedy shape. We begin in \S\ref{sec:hocolim} with a gentle introduction to this topic, illustrating examples of homotopy limit and homotopy colimits formulae provided by the Reedy model structures. In \S\ref{sec:connected-weights}, we unify all the examples just considered, proving that the limit (resp.~colimit) is a right (resp.~left) Quillen functor if and only if the weights used to define latching (resp.~matching) objects are connected. In \S\ref{sec:simplicial}, we describe the homotopical properties of the weighted limit and colimit bifunctors in a simplicial model category. A corollary is the homotopy invariance of the geometric realization of Reedy cofibrant simplicial objects and of the totalization of Reedy fibrant cosimplicial objects.

\subsection*{\S\ Acknowledgments}

During the preparation of this work the authors were supported by DARPA through the AFOSR grant number HR0011-10-1-0054-DOD35CAP and by the Australian Research Council through Discovery grant number DP1094883. The first-named author was also supported by the National Science Foundation under Award No.~DMS-1103790.

We would also like to extend personal thanks to Mike Hopkins without whose support and encouragement this work would not exist. We are also very grateful for the cogent suggestions made by the referee, which have added significantly to the quality of our exposition.

\subsection*{\S\ Notational conventions}

  \begin{ntn}[size]
    Generally speaking, in this paper matters of size will not be crucial for us. However, for definiteness we shall adopt the usual conceit of assuming that we have fixed an inaccessible cardinal which then determines a corresponding Grothendieck universe. We shall refer to the members of that universe simply as {\em sets\/} and refer to everything else as {\em classes}.

    When discussing the existence of limits and colimits we shall implicitly assume that these are indexed by small categories (categories with sets of objects and arrows). Correspondingly completeness and cocompleteness properties will implicitly reference the existence of limits and colimits indexed by small categories. To aid the intuition, we shall distinguish certain small or large structures notationally. For example, we shall usually use bold capitals $\scat{A},\scat{B},\ldots$ to denote small categories and calligraphic letters $\lcat{A},\lcat{B},\ldots$ to denote large, locally small categories. 

    While not strictly necessary in all places, we shall also adopt the blanket assumption that all of the large, locally small categories we shall consider admit all small limits and colimits.
  \end{ntn}

    \begin{ntn}[general index notation]\label{ntn:index-stuff}
In what follows, we will employ {\em index notation\/} in the context of an iterated functor category.
In this context, we use multiple subscripts and superscripts to denote the structural components from which objects and arrows are constructed. Furthermore, we follow the convention of using subscripts for contravariant actions and superscripts for covariant actions. These conventions accord with the standard notation used for $n$-fold simplicial sets and other common examples.

So for example if $X$ is an object of an iterated functor category $((\lcat{M}^{\scat{E}})^{\scat{D}\op})^{\scat{C}}$ then $X^c$ denotes the object of $(\lcat{M}^{\scat{E}})^{\scat{D}\op}$ obtained by evaluating $X$ at $c\in\scat{C}$, $X^c_d$ denotes the object of $\lcat{M}^{\scat{E}}$ obtained by evaluating $X^c$ at $d\in\scat{D}$ and finally $X^{c,e}_d$ is the object of $\lcat{M}$ obtained by evaluating $X^c_d$ at $e\in\scat{E}$. Furthermore, if $f\colon c\to \bar{c}$ is a map in $\scat{C}$ then $X^f\colon X^c\to X^{\bar{c}}$ denotes the map of $(\lcat{M}^{\scat{E}})^{\scat{D}\op}$ obtained by applying $X$ to $f$ and so forth.
		
In this context functoriality may be expressed as a family of left (covariant) and right (contravariant) actions, one for each subscript and superscript, which each satisfy the appropriate action axioms individually and which also satisfy the obvious pairwise action interchange laws. 

      It should be said that this index notation is best suited to expressing situations in which we think of the domains of our functors as being categories of operators and the functors themselves as being families of objects in the codomain category upon which those operators act. In other situations we will, of course, resort to more traditional notation for our functors.
    \end{ntn}

    \begin{ntn}[representables]
      If $\scat{C}$ is a category we will overload our notation and also use $\scat{C}$ to denote the usual two-variable hom-functor in the functor category $\Set^{\scat{C}\op\times\scat{C}}$. So using this convention, in tandem with the index conventions outlined above, we will write $\scat{C}_c$ and $\scat{C}^c$ for the covariant and contravariant representables (respectively) associated with an object $c\in\scat{C}$. So it follows that if we have a second object $\bar{c}\in\scat{C}$ then the expression $\scat{C}^c_{\bar{c}}$ denotes the value of the contravariant representable $\scat{C}^c$ at $\bar{c}$ (respectively covariant representable $\scat{C}_{\bar{c}}$ at $c$) which is simply equal to the hom-set $\scat{C}(\bar{c},c)$. We will also deploy {\em cypher notation\/} and write $\scat{C}_\bullet\colon\scat{C}\op\to\Set^{\scat{C}}$ and $\scat{C}^\bullet\colon\scat{C}\to\Set^{\scat{C}\op}$ to denote the corresponding Yoneda embeddings. 
    \end{ntn}
    
\subsection*{\S\ Weighted limits and colimits}    

Of paramount importance to enriched category are the notions of \emph{weighted limit} and \emph{weighted colimit}. As this manuscript will illustrate, these ideas can be clarifying even in the classical $\Set$-enriched case. For the reader's convenience, we include the following brief survey of (unenriched) weighted limits and colimits containing all of the facts that we will use here. A more comprehensive treatment can be found in \cite{Kelly:2005:ECT} or \cite{Riehl:2014kx}.

Ordinary limits and colimits are objects representing the $\Set$-valued functor of cones with a given summit over or under a fixed diagram. However, in the enriched context, these $\Set$-based universal properties are insufficiently expressive. The intuition is that in the presence of extra structure on the hom-sets of a category, cones over or under a diagram might come in exotic ``shapes''. 

\begin{ex}[tensors and cotensors]\label{ex:tensor-cotensor}
For example, in the case of a diagram of shape $\catone$ in a category $\lcat{M}$, the shape of a cone might be a set $S \in \Set$. Writing $X \in \lcat{M}$ for the object in the image of the diagram, the $S$-weighted limit of $X$ is an object $S \pwr X \in \lcat{M}$ satisfying the universal property \[ \lcat{M}(M, S \pwr X) \cong \Set(S,\lcat{M}(M,X))\] while the $S$-weighted colimit of $X$ is an object $S \tns X \in \lcat{M}$ satisfying the universal property \[ \lcat{M}(S \tns X, M) \cong \Set(S, \lcat{M}(X,M)).\]  For historical reasons, $S \pwr X$ is called the \emph{cotensor} and $S \tns X$ is called the \emph{tensor} of $X \in \lcat{M}$ by the set $S$. If $\lcat{M}$ has small products and coproducts, then $S\pwr X$ and $S\tns X$ are, respectively, the $S$-fold product and coproduct of the object $X$ with itself.

 Assuming the objects with these defining universal properties exist, cotensors and tensors define bifunctors \[ \pwr  \colon \Set \op \times \lcat{M} \to \lcat{M} \qquad  \tns \colon \Set \times \lcat{M} \to \lcat{M}.\]   A typical application forms tensors (or cotensors) of an object in $\lcat{M}$ with a $\Set$-valued functor of shape $\scat{C}$, producing an object in $\lcat{M}^{\scat{C}}$ (or $\lcat{M}^{\scat{C}\op}$).
\end{ex}

\begin{defn}[ends and coends]\label{defn:ends-and-coends}
  In line with our standing conventions, suppose that $\scat{C}$ is a small category and that $\lcat{M}$ is a large and locally small category which possesses all small limits and colimits. Assume also that $H\colon\scat{C}\op\times\scat{C}\to\lcat{M}$ is a bifunctor, then its {\em end\/} is the given by the equaliser
  \begin{equation*}
    \int_{c\in\scat{C}} H^c_c \defeq \mathrm{eq} \left( \prod\limits_{c \in\scat{C}} H^c_c \rightrightarrows \prod\limits_{f\colon \overline{c} \to c \in\scat{C}} H_{\overline{c}}^c\right)
  \end{equation*}
  of the obvious parallel pair of maps induced by the families:
  \begin{equation*}
    \{\xymatrix{{\prod\limits_{c \in\scat{C}} H^c_c}\ar[r]^-{\pi_c} & {H^c_c}\ar[r]^-{H^c_f} & H^c_{\overline{c}}}\}_{f\colon \overline{c}\to c\in\scat{C}}
    \mkern20mu\text{and}\mkern20mu
    \{\xymatrix{{\prod\limits_{c \in\scat{C}} H^c_c}\ar[r]^-{\pi_{\overline{c}}} & {H^{\overline{c}}_{\overline{c}}}\ar[r]^-{H^f_{\overline{c}}} & H^c_{\overline{c}}}\}_{f\colon \overline{c}\to c\in\scat{C}}
    \end{equation*}
    Dually its {\em coend\/} is the colimit given by a corresponding coequaliser:
\begin{equation*}
\int^{c\in\scat{C}} H^c_c \defeq \mathrm{coeq}\left( \coprod\limits_{\overline{c} \to c \in \scat{C}} H_c^{\overline{c}}  \rightrightarrows \coprod\limits_{c \in \scat{C}} H^c_c \right)
\end{equation*}
\end{defn}

\begin{defn}[categories of elements]\label{defn:cat-of-elts}
  Suppose that $H\colon \scat{C}\op\times\scat{D}\to\Set$ is a bifunctor. Then its {\em two sided category of elements\/} $\el H$ has 
      \begin{itemize}
         \item objects triples $(c,d,x)$ where $c$ is an object of $\scat{C}$, $d$ is an object of $\scat{D}$, and $x$ is an element of $H^d_c$, and
         \item arrows $(f,g)\colon(c,d,x)\to(\bar{c},\bar{d},\bar{x})$ which comprise pairs of arrows $f\colon c\to \bar{c}$ in $\scat{C}$ and $g\colon d\to \bar{d}$ with the property that $H^g_d(x) = H^{\bar{c}}_f(\bar{x})$, whose composite and identities are as in $\scat{C}$ and $\scat{D}$.
       \end{itemize}
       This category comes equipped with an obvious forgetful functor $\el{H}\to\scat{C}\times\scat{D}$. This construction specialises to give category of elements constructions for covariant and contravariant functors $F\colon \scat{D}\to\Set$ and $G\colon\scat{C}\op\to\Set$.
   \end{defn}

\begin{obs}[coends in sets]\label{obs:coend-in-sets}
  Suppose that $H\colon\scat{C}\op\times\scat{C}\to\Set$ is a bifunctor. Then its coend $\int^{c\in\scat{C}} H^c_c$ admits a concrete description in terms of the category of elements construction. Specifically, we may form the following pullback of categories
  \begin{equation}\label{eq:pb-coends-in-sets}
    \xymatrix@=2em{
      {d^*(\el{H})}\pbexcursion
      \ar[r]\ar[d] & {\el{H}}\ar[d] \\
      {\scat{C}}\ar[r]_-{d} &
      {\scat{C}\times\scat{C}}
    }
  \end{equation}
  in which $d\colon\scat{C}\to\scat{C}\times\scat{C}$ denotes the diagonal functor. Then we may show, directly from the defining coequaliser in definition~\ref{defn:ends-and-coends}, that $\int^{c\in\scat{C}} H^c_c$ is canonically isomorphic to the set $\pi_0(d^*(\el{H}))$ of connected components of the category appearing in that pullback. More explicitly, the category $d^*(\el{H})$ has: 
  \begin{itemize}
    \item objects which are pairs $(c,x)$ in which $c$ is an object of $\scat{C}$ and $x$ is an element of $H^c_c$, and 
    \item arrows $f\colon(c,x)\to(\bar{c},\bar{x})$ which are arrows $f\colon c\to\bar{c}$ of $\scat{C}$ satisfying the equation $H_c^f(x) = H_f^{\bar{c}}(\bar{x})$, whose composites and identities are as in $\scat{C}$.
  \end{itemize}
  Its set of connected components is isomorphic to the set of equivalence classes $[c,x]$ of objects of $d^*(\el{H})$ wherein two objects are equivalent if and only if they are connected by a finite zig-zag of arrows in there.
\end{obs}

\begin{defn}[weighted limits and colimits] 
 Suppose $\scat{C}$ is a small category and that $\lcat{M}$ is a large and locally small category which is (small) complete and cocomplete. In this context, the weighted limit 
      \begin{equation}\label{eq:wlimformula}
      \wlim{\,}{\,}_{\scat{C}}\colon (\Set^{\scat{C}})\op\times\lcat{M}^{\scat{C}} \to {\lcat{M}}
      \mkern30mu\text{where}\mkern30mu
      \wlim{W}{X}_{\scat{C}} \defeq \int_{c\in\scat{C}} W^c\pwr X^c
    \end{equation}    and  weighted colimit
    \begin{equation*}
      \wcolim_{\scat{C}}\colon \Set^{\scat{C}\op}\times\lcat{M}^{\scat{C}} \to {\lcat{M}}
      \mkern30mu\text{where}\mkern30mu
      V\wcolim_{\scat{C}} X \defeq \int^{c\in\scat{C}} V_c\tns X^c
    \end{equation*}
define bifunctors.  We refer to the object $\wlim{W}{X}_\scat{C}$ as the \emph{limit} of the diagram $X\colon \scat{C}\to\lcat{M}$ \emph{weighted by} $W\colon\scat{C}\to\Set$ and to $V\wcolim_\scat{C}X$ as the {\em colimit\/} of $X$ \emph{weighted by} $V\colon\scat{C}\op\to\Set$. These objects are characterised by the universal properties \[ \lcat{M}(M, \wlim{W}{X}_{\scat{C}}) \cong \Set^{\scat{C}}(W, \lcat{M}(M,X)) \qquad \lcat{M}(V\wcolim_{\scat{C}} X, M) \cong \Set^{\scat{C}\op}(V,\lcat{M}(X,M)).\] 

A natural transformation of weights $f\colon\overline{W} \to W$ in $\Set^{\scat{C}}$ induces a derived morphism $\wlim{f}{X}_\scat{C}\colon\wlim{W}{X}_\scat{C} \to \wlim{\overline{W}}{X}_\scat{C}$ between weighted limits and a natural transformation $g\colon V\to \overline{V}$ in $\Set^{\scat{C}\op}$ induces a derived morphism $g\wcolim_{\scat{C}} X\colon V\wcolim_{\scat{C}} X \to \overline{V}\wcolim_{\scat{C}} X$ between weighted colimits.
\end{defn}
    
  \begin{ex}[terminal weights]\label{ex:weighted-conical}
  When $W$ is, respectively, the constant $\scat{C}$-diagram or $\scat{C}\op$-diagram at the terminal object $1 \in \Set$, we see from the defining formulae \eqref{eq:wlimformula} that \[ \wlim{1}{X}_\scat{C} \cong \lim X \qquad \mathrm{and} \qquad 1 \wcolim_{\scat{C}} X \cong \colim X.\] Here the weight $1$ stipulates that the cones should have their ordinary ``shapes'' with one leg pointing to or from each object in the diagram $X$.
  \end{ex}  
    
\begin{ex}[representable weights]\label{ex:weighted-yoneda}
The bifunctors \eqref{eq:wlimformula} admit canonical isomorphisms
    \begin{equation}\label{eq:yoneda}
     \wlim{\scat{C}_a}{X}_{\scat{C}} \cong \int_{b\in\scat{C}} \scat{C}(a,b)\pwr X^b \cong X^a
\mkern40mu
     \scat{C}^a\wcolim_{\scat{C}} X \cong \int^{b\in\scat{C}} \scat{C}(b,a)\tns X^b \cong X^a 
    \end{equation}
    which are natural in $a\in\scat{C}$ and $X\in\lcat{M}^{\scat{C}}$. This result is simply a recasting of the classical Yoneda lemma, which name it bears herein. Hence, limits or colimits weighted by representables are computed simply by evaluating the diagram at the appropriate object.
    \end{ex}
    
    \begin{obs} On account of their explicit construction or their defining universal property, the weighted limit and weighted colimit bifunctors are each cocontinuous in the weights. In particular weights can be ``made-to-order''. A weight constructed as a colimit of representables will stipulate the expected universal property. 
   \end{obs}

\begin{ex}\label{ex:matching.preview}
Let $K$ and $X$ be simplicial sets, i.e., objects of the functor category $\Set^{\Del\op}$ where $\Del$ is the category of finite non-empty ordinals and order-preserving maps. From the end formula \eqref{eq:wlimformula}, the limit of $X$ weighted by $K$ is merely the set of natural transformations $K \to X$, i.e., the set of maps from $K$ to $X$ in the category of simplicial sets. For instance,  the weighted limit $\wlim{\partial\Delta^n}{X}_{\Del\op}$ is the set of $n$-spheres in $X$.
\end{ex}

\begin{ex}
Applying the coend formula \eqref{eq:yoneda} pointwise, one sees that the colimit of the Yoneda embedding $\Del^\bullet\colon\Del\inc\Set^{\Del\op}$ weighted by a simplicial set $X$ is isomorphic to $X$.
\end{ex}

\begin{obs}[weighted limits as ordinary limits]\label{obs:weighted.as.ordinary}
  Suppose $D$ is a diagram $\lcat{M}^{\scat{C}}$ and $W$ is a weight in $\Set^{\scat{C}}$ then the weighted limit $\wlim{W}{D}_\scat{C}$ is equally the ordinary limit of the composite $\el W \to \scat{C} \xrightarrow{D} \lcat{M}$. Dually if $V$ is a weight in $\Set^{\scat{S}\op}$ then the weighted colimit $V \wcolim_\scat{C}D$ is the ordinary colimit of the composite diagram $\el V \to \scat{C} \xrightarrow{D} \lcat{M}$.
\end{obs}

\begin{ex} The category of elements of $\scat{C}^c \in \Set^{\scat{C}\op}$ is isomorphic to the slice category $\scat{C}/c$. This category has a terminal object: the identity at $c$. Hence, the colimit of $\el \scat{C}^c \to \scat{C} \xrightarrow{D} \lcat{M}$ is $Dc$, recovering the fact observed above that $Dc \cong \scat{C}^c \wcolim_\scat{C} D$. Dual remarks apply to covariant representables.
\end{ex}

  \begin{obs}[weighted colimits in sets]\label{obs:wcolim-in-sets}
    Suppose that $W\colon\scat{C}\op\to\Set$ is a weight and that $D\colon\scat{C}\to\Set$ is a diagram of sets. We may apply observation~\ref{obs:coend-in-sets} to provide an explicit computation of the defining coend of the weighted colimit $W\wcolim_{\scat{C}} D$. In this case, the coend we are computing is $\int^{c\in\scat{C}} W_c\times D^c$ and it is easily checked that the corresponding category given in~\eqref{eq:pb-coends-in-sets} may also be constructed by forming the following pullback
    \begin{equation*}
      \xymatrix@=2em{
        {\el{W}\times_{\scat{C}}\el{D}}\pbexcursion
        \ar[r]\ar[d] & {\el{D}}\ar[d] \\
        {\el{W}}\ar[r] & {\scat{C}}
      }
    \end{equation*}
    of categories. In other words, this has:
    \begin{itemize}
      \item objects which are triples $(c,x,y)$ in which $c$ is an object of $\scat{C}$, $x$ is an element of $W_c$, and $y$ is an element of $D^c$, and
      \item arrows $f\colon(c,x,y)\to(\bar{c},\bar{x},\bar{y})$ which consist of an arrow $f\colon c\to\bar{c}$ of $\scat{C}$ satisfying the equations $W_f(\bar{x})=x$ and $D^f(y)=\bar{y}$, whose composites and identities are as in $\scat{C}$.
    \end{itemize}
    So $W\wcolim_{\scat{C}} D$ is isomorphic to the set of equivalence classes $[c,x,y]$ of objects in this category, wherein two objects are equivalent if and only if they are connected by a zig-zag of maps in there.
  \end{obs}

  \section{Reedy categories}\label{sec:reedy}

  Let $\mathbb{N}$ denote the natural numbers, including zero.

  \begin{defn}[Reedy categories]\label{defn:reedy}
    Let $\scat{C}$ be a small category, equipped with a {\em degree function\/} $\deg\colon\obj(\scat{C})\to\mathbb{N}$ on objects, and suppose that $\direct{\scat{C}}$ and $\inverse{\scat{C}}$ are subcategories of $\scat{C}$ which contain all of its objects. Then we say that $\reedycat{C}$ is a {\em Reedy category\/} if and only if these structures satisfy:
    \begin{itemize}
    \item if $\alpha\colon \bar{c}\to c$ is a {\em non-identity\/} arrow in $\direct{\scat{C}}$ (respectively $\inverse{\scat{C}}$) then $\deg(\bar{c})<\deg(c)$ (respectively $\deg(\bar{c})>\deg(c)$), and
    \item every arrow $\alpha$ of $\scat{C}$ has a {\em unique\/} factorisation $\alpha=\direct\alpha\circ\inverse\alpha$ where $\direct\alpha\in\direct{\scat{C}}$ and $\inverse\alpha\in\inverse{\scat{C}}$.
    \end{itemize}

We say that an arrow \emph{strictly raises} (respectively \emph{strictly lowers}) \emph{degrees} precisely when it is a  non-identity arrow of $\direct{\scat{C}}$ (respectively of $\inverse{\scat{C}}$).
  \end{defn}
	
\begin{ex}[discrete categories] Discrete categories are Reedy categories. Our preference is to regard every object as having degree zero.
\end{ex}

\begin{ex}[finite posets]\label{ex:poset.reedy}
Let $\scat{C}$ be a finite poset. Declare any minimal element to have degree zero.
Define the degree of $c \in \scat{C}$ to be the length of the longest path of non-identity arrows from an element of degree zero to $c$. Note that if there is a morphism $c \to \bar{c}$ between distinct elements, then necessarily $\deg{c} < \deg{\bar{c}}$. It follows that $\scat{C}$ has the structure of a Reedy category with $\scat{C} = \direct{\scat{C}}$ and $\inverse{\scat{C}}$ consisting of identity arrows only.

This example can be extended without change to include infinite posets such as $(\omega,\leq)$ provided that each object has finite degree.
\end{ex}

\begin{ex}[the pushout diagram]\label{ex:pushout.reedy.alt}
The previous example gives the category $b \leftarrow a \rightarrow c$ a Reedy structure in which $\deg(a)=0$ and $\deg(b)= \deg(c)=1$. There is another Reedy category structure in which $\deg(b)=0$, $\deg(a)=1$, and $\deg(c)=2$. \end{ex}

\begin{ex}[the parallel pair]\label{ex:parallel.pair.reedy}
The category $a \rightrightarrows b$ is a Reedy category with $\deg(a)=0$, $\deg(b)=1$, and both non-identity arrows said to strictly raise degrees.
\end{ex}
	
	\begin{ex}[\protect{$\Del$} and \protect{$\Del+$}]
		The category $\Del$ of finite non-empty ordinals and the category $\Del+$ of finite ordinals and order-preserving maps  both support canonical Reedy category structures, for which we take $\direct{\Del}$ and $\direct{\Del+}$ to be the subcategories of face operators (monomorphisms) and $\inverse{\Del}$ and $\inverse{\Del+}$ to be the subcategories of degeneracy operators (epimorphisms).  	
		\end{ex}

  \begin{obs}
    Many fundamental operations on Reedy categories yield Reedy categories, most importantly:
      \begin{itemize}
      \item If $\reedycat{C}$ is a Reedy category then so is its dual $(\scat{C}\op,(\inverse{\scat{C}})\op,(\direct{\scat{C}})\op)$.
      \item If $\reedycat{C}$ and $\reedycat{D}$ are Reedy categories then so is the product $(\scat{C}\times\scat{D}, \direct{\scat{C}}\times\direct{\scat{D}}, \inverse{\scat{C}}\times\inverse{\scat{D}})$ with $\deg(c,d) = \deg(c) + \deg(d)$.
      \end{itemize}
		In particular, these observations tell us that $\Del\op$ and $\Del+\op$ also carry canonical Reedy structures as do arbitrary products of these with themselves and with $\Del$ and $\Del+$.
  \end{obs}

Given an arrow $f \colon \bar{c} \to c$ in a Reedy category $\scat{C}$, we shall call the unique factorisation $(\inverse{f},\direct{f})$ stipulated by definition \ref{defn:reedy} the {\em canonical Reedy factorisation\/} of $f$. Our aim, realised in lemma \ref{lem:reedy-fact-connect}, is to show that the canonical Reedy factorisation may also  be characterised as the unique factorisation of $f$ through an object of minimal degree. 

  \begin{defn}[categories of factorisations]\label{defn:factorisations}
To that end, for each arrow $f\colon \bar{c}\to c\in\scat{C}$, we define the category $\fact(f)$ of {\em factorisations of $f$} to be the category whose objects are pairs $(g\colon \bar{c}\to d, h\colon d\to c)$ of arrows of $\scat{C}$ with $f=h\circ g$ and whose arrows $k\colon (g,h)\to (g',h')$ are arrows $k\colon d\to d'$ which make the following triangles commute:
    \begin{equation*}
    \xy
      0;<1cm,0cm>: (1,2)*+{\bar{c}}="a", (0,1)*+{d}="b", (2,1)*+{d'}="c", (1,0)*+{c}="d",
      \ar "a";"b"_g \ar "a";"c"^{g'} \ar "b";"d"_h \ar "c";"d"^{h'} \ar "b";"c"^{k}
    \endxy
    \end{equation*}
For each object $(g,h) \in \fact(f)$,  we shall take its {\em degree\/} $\deg(g,h)$ to simply be the degree of the intermediate object $\cod(g)=\dom(h)$. 
  \end{defn}
  
  Now we have the following result, which provides much more information about the structure of the categories of factorisations of our Reedy category $\scat{C}$. 
  
  \begin{lem}\label{lem:reedy-fact-connect}
    Every factorisation $(g,h)$ of $f$ is connected to the canonical Reedy factorisation $(\inverse{f},\direct{f})$ by a zig-zag path of maps in the category $\fact(f)$ which passes only through factorisations of degree greater than or equal to that of $(\inverse{f},\direct{f})$ and less than or equal to the degree of $(g,h)$ itself. Furthermore, the Reedy factorisation $(\inverse{f},\direct{f})$ is the unique factorisation of minimal degree in $\fact(f)$.
  \end{lem}

  \begin{proof}
    Observe that we may (repeatedly) apply the factorisation property of the Reedy category $\scat{C}$ to obtain the following commutative diagram:
    \begin{equation*}
    \xy
      0,<1.5cm,0cm>:
      (0,0)*+{\bar{c}}="(0,0)", (2,0)*+{d}="(2,0)", (4,0)*+{c}="(4,0)",
      (1,1)*+{\bar{a}}="(1,1)", (3,1)*+{a}="(3,1)", (2,2)*+{b}="(2,2)"
      \ar "(0,0)";"(2,0)"^g \ar "(2,0)";"(4,0)"^h
      \ar "(0,0)";"(1,1)"|{\inverse{g}} \ar "(1,1)";"(2,0)"|{\direct{g}}
      \ar "(2,0)";"(3,1)"|{\inverse{h}} \ar "(3,1)";"(4,0)"|{\direct{h}}
      \ar "(1,1)";"(2,2)"|{\inverse{k}} \ar "(2,2)";"(3,1)"|{\direct{k}}
      \ar "(1,1)";"(3,1)"^{k \, = \, \inverse{h}\circ\direct{g}}
      \ar@/^2pc/"(0,0)";"(2,2)"^{\inverse{f}}
      \ar@/^2pc/"(2,2)";"(4,0)"^{\direct{f}}
      \ar@/_2pc/"(0,0)";"(4,0)"_{{f}}
    \endxy
    \end{equation*}
    Now $\inverse{k}\circ\inverse{g}$ is in the subcategory $\inverse{\scat{C}}$ and $\direct{h}\circ\direct{k}$ is in the subcategory $\direct{\scat{C}}$, because these are composites of arrows already known to be in those subcategories, so it follows, by the uniqueness property of the Reedy factorisation of $f$ in $\scat{C}$, that $\inverse{f}=\inverse{k}\circ\inverse{g}$ and $\direct{f}=\direct{h}\circ\direct{k}$. Also observe that $\deg(\bar{a})\leq\deg(d)$ because $\direct{g}\in\direct{\scat{C}}$, $\deg(a)\leq\deg(d)$ because $\inverse{h}\in\inverse{\scat{C}}$, $\deg(b)\leq\deg(\bar{a})$ because $\inverse{k}\in\inverse{\scat{C}}$ and $\deg(b)\leq\deg(a)$ because $\direct{k}\in\direct{\scat{C}}$. Furthermore, these inequalities are strict unless the maps are identities. In particular, $\deg(b) < \deg(d)$ unless $(g,h) = (\inverse{f},\direct{f})$, proving the last statement.
    
 Examining the diagram, it follows that $f$ has factorisations $(g,h)$, $(\inverse{g},\direct{h}\circ k)$, $(k\circ\inverse{g},\direct{h})$, for which $\deg(\inverse{f},\direct{f}) \leq \deg(\inverse{g},\direct{h}\circ k)\leq \deg(g,h)$ and $\deg(\inverse{f},\direct{f}) \leq \deg(k\circ\inverse{g},\direct{h}) \leq \deg(g,h)$, and that these are connected by arrows $\direct{g}\colon(\inverse{g},\direct{h}\circ k)\to(g,h)$, $\inverse{h}\colon(g,h)\to(k\circ\inverse{g},\direct{h})$, $k\colon(\inverse{g},\direct{h}\circ k)\to(k\circ\inverse{g},\direct{h})$, $\inverse{k}\colon(\inverse{g},\direct{h}\circ k)\to(\inverse{f},\direct{f})$, and $\direct{k}\colon(\inverse{f},\direct{f})\to(k\circ\inverse{g},\direct{h})$ in $\fact(f)$ as required. 
  \end{proof}

  \section{Latching and matching objects}\label{sec:latching}

Our description of the inductive procedure by which a diagram indexed over a Reedy category may be defined, which is instrumental for the characterisation of the homotopical properties of limits and colimits, relies upon the notions of {\em latching\/} and {\em matching\/} objects. These definitions in turn rely upon the notions of skeleta and coskeleta, which we now review.

\subsection*{\S\ Skeleta and coskeleta}

  \begin{rec}[skeleta and coskeleta]\label{rec:skel-coskel}
    Let $\scat{D}$ be a full subcategory of a small category $\scat{C}$ and let $\lcat{M}$ be a complete and cocomplete category. Then we know that pre-composition with the inclusion functor $\scat{D}\inc\scat{C}$ gives rise to a functor $\lcat{M}^\scat{C} \to \lcat{M}^\scat{D}$   
    which in this context we refer to as {\em restriction\/} or occasionally as {\em truncation\/} and which has left and right adjoints, called {\em skeleton\/} and {\em coskeleton\/} respectively:
         \begin{equation*}
      \xymatrix@R=0ex@C=12em{
      {\lcat{M}^{\scat{C}}}\ar[r]^{}="two"|*+{\labelstyle\res}_{}="three" &
      {\lcat{M}^{\scat{D}}}
      \ar@/_3ex/[l]_{\sk}^{}="one"
      \ar@/^3ex/[l]^{\cosk}_{}="four"
      \ar@{}"one";"two"|{\textstyle\bot}
      \ar@{}"three";"four"|{\textstyle\bot}
      }
    \end{equation*}     
    These adjoints are simply (pointwise) left and right Kan extension along $\scat{D}\inc\scat{C}$ respectively, and may thus be given, for $X \in \lcat{M}^{\scat{D}}$, by the coend and end formulae:
    \begin{equation}\label{eq:4}
    \sk(X)^c \cong \int^{d\in\scat{D}} \scat{C}(d,c)\tns X^d \mkern40mu
    \cosk(X)^c \cong \int_{d\in\scat{D}} \scat{C}(c,d)\pwr X^d.
    \end{equation}
    Here the symbols $\tns\colon \Set\times\lcat{M}\to\lcat{M}$ and $\pwr\colon\Set\op\times\lcat{M}\to\lcat{M}$ denote the tensor and cotensor bifunctors introduced in example \ref{ex:tensor-cotensor}.
    
  Because $\scat{D}\inc\scat{C}$ is fully faithful, the functors $\sk$ and $\cosk$ are too, which means that the unit of $\sk\dashv\res$ and the counit of $\res\dashv\cosk$ are invertible. Hence, there exists a natural transformation $\tau\colon\sk\to\cosk$ which is equal to both of the composites
    \begin{equation*}
    \sk\xrightarrow{\unit^c\circ\sk}\cosk\circ\res\circ\sk
    \xrightarrow{\cosk\circ(\unit^s)^{-1}}\cosk
    \end{equation*}
    and
    \begin{equation*}
    \sk\xrightarrow{\sk\circ(\counit^c)^{-1}}\sk\circ\res\circ\cosk
    \xrightarrow{\counit^s\circ\cosk}\cosk.
    \end{equation*}
    Here $\unit^s$ and $\counit^s$ are the unit and counit (respectively) of the adjunction $\sk\dashv\res$,  and $\unit^c$ and $\counit^c$ are the unit and counit (respectively) of the adjunction $\res\dashv\cosk$.
\end{rec}

\begin{obs}\label{obs:skel-colim}
We may recast the formulae of \eqref{eq:4} as weighted limits and colimits:
    \begin{equation}\label{eq:5}
    \sk(X)^c \cong \res(\scat{C}^c)\wcolim_{\scat{D}} X \mkern40mu
    \cosk(X)^c \cong \wlim{\res(\scat{C}_c)}{X}_{\scat{D}}.
    \end{equation}
    Notice here that the representable $\scat{C}^c$ is contravariant, so in the expression $\res(\scat{C}^c)$ of~\eqref{eq:5} the symbol ``$\res$'' denotes the restriction functor associated with the dual inclusion $\scat{D}\op\inc\scat{C}\op$. 

    Now suppose that $d$ is an object of the full subcategory $\scat{D}$. Because the inclusion $\scat{D} \inc \scat{C}$ is full it is clear that $\res(\scat{C}^d)=\scat{D}^d$, and so we may apply Yoneda's lemma  \ref{ex:weighted-yoneda} to show that:
    \begin{gather*}
    \sk(X)^d \cong \res(\scat{C}^d)\wcolim_{\scat{D}} X = \scat{D}^d\wcolim_{\scat{D}} X \cong X^d \\
    \cosk(X)^d \cong \wlim{\res(\scat{C}_d)}{X}_{\scat{D}} = \wlim{\scat{D}_d}{X}_{\scat{D}} \cong X^d.
    \end{gather*}
    Indeed, this computation demonstrates that the unit $\unit^s_X\colon X\to\res(\sk(X))$ of the adjunction $\sk\dashv\res$ and the counit $\counit^c_X\colon\res(\cosk(X))\to X$ of the adjunction $\res\dashv\cosk$ are both isomorphisms or, equivalently, that the functors $\sk$ and $\cosk$ are both fully faithful. 
    \end{obs}

    From hereon we will adopt the traditional convention of blurring the distinction between the functor $\sk$ (respectively $\cosk$), which maps $\lcat{M}^{\scat{D}}$ to $\lcat{M}^{\scat{C}}$, and the corresponding endo-functor $\sk\circ\res$ (respectively $\cosk\circ\res$) on $\lcat{M}^{\scat{C}}$. So we will write $\sk$ (respectively $\cosk$) for either of these and allow the context to disambiguate each instance.

  \begin{lem}\label{lem:skeleta-colimit}
Let $X \in \lcat{M}^{\scat{C}}$ and let $c \in \scat{C}$. The object $\sk(X)^c$ is the colimit of $X$ weighted by $\sk(\scat{C}^c)$ and the object $\cosk(X)^c$ is the limit of $X$ weighted by $\sk(\scat{C}_c)$, i.e., there exist isomorphisms:
    \begin{equation}\label{eq:6}
    \sk(X)^c \cong \sk(\scat{C}^c)\wcolim_{\scat{C}} X \mkern40mu
    \cosk(X)^c \cong \wlim{\sk(\scat{C}_c)}{X}_{\scat{C}}
    \end{equation}
    which are natural in $X\in\lcat{M}^{\scat{C}}$ and $c\in\scat{C}$. 
    \end{lem}
 
 Lemma \ref{lem:skeleta-colimit} is a special case of a general result: the weighted limit or colimit of the restriction of a diagram is isomorphic to the limit or colimit of the original diagram weighted by the left Kan extension of the weight.

    \begin{proof} The isomorphisms \eqref{eq:6} are constructed in the following calculation (and its dual):
    \begin{align*}
    \sk(X)^c  &\cong \int^{d\in \scat{D}} \scat{C}(d,c)\tns X^d && 
          \text{\eqref{eq:4}}\\
          &\cong \int^{d\in \scat{D}} \scat{C}(d,c)\tns\left( \int^{\bar{c}\in \scat{C}}
          \scat{C}(\bar{c},d)\tns X^{\bar{c}}\right) && 
          \text{by Yoneda's lemma} \\
          &\cong \int^{\bar{c}\in\scat{C}} \left(\int^{d\in\scat{D}}\scat{C}(d,c)\times
          \scat{C}(\bar{c},d)\right)\tns X^{\bar{c}} &&
          \text{cocontinuity of $\tns$ and Fubini} \\
          &\cong \sk(\scat{C}^c)\wcolim_{\scat{C}} X &&
          \text{definitions of $\sk$ and $\wcolim_{\scat{C}}$}\qedhere
    \end{align*}
    \end{proof}
    
    \begin{obs}[skeleta of representables and factorisations]\label{obs:skel-reps}
    Given the importance of the skeleta of the representables $\scat{C}^c$ and $\scat{C}_{\bar{c}}$ in these expressions, it will be useful to analyse these in a little more detail. To that end, observe that formula~\eqref{eq:4} tells us that the skeleton $\sk(\scat{C}_{\bar{c}})$ is given by the following coend formula;
    \begin{equation}\label{eq:sk-rep-coend}
      \sk(\scat{C}_{\bar{c}})^c \cong 
      \int^{d\in\scat{D}} \scat{C}(d,c)\times
      \scat{C}(\bar{c},d)
    \end{equation}
    Furthermore, the counit map $\sk(\scat{C}_{\bar{c}})\to\scat{C}_{\bar{c}}$ is induced, via the universal property of that coend, by the family of composition maps
    \[ \xymatrix{ {\scat{C}(d,c)\times\scat{C}(\bar{c},d)} \ar[r]^-{\circ} & \scat{C}(\bar{c},c)}\] 
    which are natural in $c,\bar{c}\in\scat{C}$ and dinatural in $d\in\scat{D}$. It will be useful to know when this counit map is a monomorphism. 
   
   Applying observation~\ref{obs:coend-in-sets} to the coend in~\eqref{eq:sk-rep-coend} we see that it is isomorphic to the set of connected components of a category $\fact_{\scat{C}}(\bar{c},c)$ which has
   \begin{itemize}
     \item objects triples $(d,g,h)$ comprising an object $d$ of $\scat{D}$ and a pair of arrows $g\colon\bar{c}\to d$ and $h\colon d\to c$ in $\scat{C}$, and
     \item arrows $k\colon (d,g,h)\to(d',g',h')$ which consist of an arrow $k\colon d\to d'$ of $\scat{D}$ making the following triangles commute
    \begin{equation*}
    \xy
      0;<1cm,0cm>: (1,2)*+{\bar{c}}="a", (0,1)*+{d}="b", (2,1)*+{d'}="c", (1,0)*+{c}="d",
      \ar "a";"b"_g \ar "a";"c"^{g'} \ar "b";"d"_h \ar "c";"d"^{h'} \ar "b";"c"^{k}
    \endxy
    \end{equation*}
    whose composition and identities are as in $\scat{D}$.
   \end{itemize}
  Under this presentation of that coend the counit map $\sk(\scat{C}_{\bar{c}})^c\to\scat{C}_{\bar{c}}^c$ carries an equivalence class (connected component) $[d,g,h]$ to the composite map $hg\colon\bar{c}\to c$. 

  Observe that the category $\fact_{\scat{C}}(\bar{c},c)$ discussed in the last paragraph splits up into a disjoint union $\coprod_{f\in\scat{C}(\bar{c},c)}\fact_{\scat{D}}(f)$, where $\fact_{\scat{D}}(f)$ is its full sub-category determined by those objects $(d,g,h)$ for which $hg = f$. We call $\fact_{\scat{D}}(f)$ the category of {\em factorisations of $f$ through $\scat{D}$}. Then it is clear that the set $\pi_0(\fact_{\scat{D}}(f))$ of connected components of that category of factorisations is the subset of $\sk(\scat{C}_{\bar{c}})^c\cong\pi_0(\fact_{\scat{C}}(\bar{c},c))$ which is mapped to the element $f\in\scat{C}_{\bar{c}}^c$ under the action of the counit $\sk(\scat{C}_{\bar{c}})^c\to\scat{C}_{\bar{c}}^c$. The dual result for $\sk(\scat{C}^c)$ tells us that the fibre of the counit $\sk(\scat{C}^{c})_{\bar{c}}\to\scat{C}^{c}_{\bar{c}}$ over $f\in\scat{C}^{c}_{\bar{c}}$ may also be described as the set of connected components of $\fact_{\scat{D}}(f)$.

  It follows therefore that counit $\sk(\scat{C}_{\bar{c}})\to\scat{C}_{\bar{c}}$ (respectively $\sk(\scat{C}^{c})\to\scat{C}^{c}$) is a monomorphism if and only if for each arrow $f\colon \bar{c}\to c$ in $\scat{C}$ the category $\fact_{\scat{D}}(f)$ is either empty or connected. Furthermore, an arrow $f\colon \bar{c}\to c\in\scat{C}$ is in the image of $\sk(\scat{C}_{\bar{c}})\to\scat{C}_{\bar{c}}$ (respectively $\sk(\scat{C}^{c})\to\scat{C}^{c}$) if and only if $\fact_{\scat{D}}(f)$ is non-empty.
  \end{obs}
  
 \subsection*{\S\ Latching and matching objects} 
  
  We now explore the consequences of these observations in the Reedy setting. Henceforth, we shall assume that $\scat{C}$ is a Reedy category with a given Reedy structure. Let $\reedyfilt\scat{C}_n$ (for $n\in\mathbb{N}$) denote the full subcategory of $\scat{C}$ whose objects are those $c\in\scat{C}$ with $\deg(c)\leq n$. 

  \begin{obs}[skeleta and coskeleta in a Reedy setting]
    Suppose that $\lcat{M}$ is a category which possesses all limits and colimits. Then just as in recollection~\ref{rec:skel-coskel}, we obtain a pair of adjunctions:
    \begin{equation*}
    \xymatrix@R=0ex@C=12em{
      {\lcat{M}^{\scat{C}}}\ar[r]^{}="two"|{\res_n}_{}="three" &
      {\lcat{M}^{\reedyfilt\scat{C}_n}}
      \ar@/_3ex/[l]_{\sk_n}^{}="one"
      \ar@/^3ex/[l]^{\cosk_n}_{}="four"
      \ar@{}"one";"two"|{\textstyle\bot}
      \ar@{}"three";"four"|{\textstyle\bot}
    }
    \end{equation*}
    called {\em $n$-skeleton}, {\em $n$-truncation\/}, and {\em $n$-coskeleton\/} respectively. As the inclusion $\reedyfilt\scat{C}_n\inc\scat{C}$ is fully faithful, there exists a natural transformation $\tau_n\colon\sk_n\to\cosk_n$,  defined as in recollection \ref{rec:skel-coskel} from the units and counts $\eta^s_n$, $\eta^c_n$, $\epsilon^s_n$, and $\epsilon^c_n$ of the adjunctions $\sk_n\dashv\res_n\dashv\cosk_n$.
    \end{obs}

    The maps $\tau_n\colon\sk_n\to\cosk_n$ are of particular interest to us here since we may show that each extension of a diagram $X \in \lcat{M}^{\reedyfilt\scat{C}_{n-1}}$ to an object of $\lcat{M}^{\reedyfilt\scat{C}_n}$ corresponds to and is uniquely determined by a family of factorisations $\sk_n(X)^c\stackrel{i^c}\to X^c\stackrel{p^c}\to\cosk_n(X)^c$ of the maps $(\tau_{n-1})^{X,c}\colon\sk_{n-1}(X)^c\to\cosk_{n-1}(X)^c$ for each $c\in\obj(\scat{C})$ with $\deg(c)=n$.

    \begin{lem}[inductive definition of diagrams]\label{lem:inductive.defn} A diagram $X \in \lcat{M}^{\reedyfilt\scat{C}_{n-1}}$ together with a family of factorisations $\sk_n(X)^c\stackrel{i^c}\to X^c\stackrel{p^c}\to\cosk_n(X)^c$ of the maps $(\tau_{n-1})^{X,c}\colon\sk_{n-1}(X)^c\to\cosk_{n-1}(X)^c$ for each $c\in\obj(\scat{C})$ with $\deg(c)=n$, uniquely determines a diagram $X \in \lcat{M}^{\reedyfilt\scat{C}_n}$ whose restriction to degree $n-1$ coincides with the original diagram. 
    \end{lem}
    \begin{proof}
    It remains to define the action of $X$ on (non-identity) morphisms whose domain or codomain has degree $n$. Given such a map $f \colon \bar{c} \to c$, its Reedy factorisation $(\inverse{f},\direct{f})$ is through an object of degree less than $n$. There exist unique dotted-arrow maps making the following diagram commute
        \begin{equation*}
    \xymatrix@R=2.5em@C=5em{
      {\sk_{n-1}(X)^{\bar{c}}}\ar[r]^-{i^{\bar{c}}}
      \ar[d]_{\sk_{n-1}(X)^{\inverse{f}}} &
      {X^{\bar{c}}}\ar[r]^-{p^{\bar{c}}}
      \ar@{-->}[d]_{X^{\inverse{f}}} &
      {\cosk_{n-1}(X)^{\bar{c}}}\ar[d]^{\cosk_{n-1}(X)^{\inverse{f}}} \\
      {\sk_{n-1}(X)^d}\ar[r]_-{=}^-{i^{d}} 
         \ar[d]_{\sk_{n-1}(X)^{\direct{f}}} &
     {X^d}\ar[r]_-{=}^-{p^d}       \ar@{-->}[d]_{X^{\direct{f}}}&
      {\cosk_{n-1}(X)^d}
      \ar[d]^{\cosk_{n-1}(X)^{\direct{f}} }     \\
           {\sk_{n-1}(X)^{d}}\ar[r]^-{i^{c}}
    &
      {X^{c}}\ar[r]^-{p^{c}}
      &
      {\cosk_{n-1}(X)^{c}}
    }
    \end{equation*}
    defined to be the composites of the maps in the upper-right and lower-left squares respectively. The functoriality of this definition, in a pair of composable maps $(f,g)$ follows from connectedness of the category $\fact_{n-1}(gf)$. 
    \end{proof}
    
    \begin{obs}[inductive definition of natural transformations]\label{obs:inductive.defn}
        More specifically, we may apply this result to show that if $X$ and $Y$ are objects of $\lcat{M}^{\scat{C}}$ then each extension of a natural transformation $\phi\colon \res_{n-1}(X)\to\res_{n-1}(Y)$ to a natural transformation $\phi\colon\res_n(X)\to\res_n(Y)$ corresponds to a unique family of maps $\{\phi^c\colon X^c\to Y^c \mid c\in\obj(\scat{C}), \deg(c)=n\}$ which make  the following diagrams commute:
    \begin{equation*}
    \xymatrix@R=2.5em@C=5em{
      {\sk_{n-1}(\res_{n-1}(X))^c}\ar[r]^-{(\counit^s_{n-1})^{X,c}}
      \ar[d]_{\sk_{n-1}(\phi)^c} &
      {X^c}\ar[r]^-{(\unit^c_{n-1})^{X,c}}
      \ar@{-->}[d]_{\phi^c} &
      {\cosk_{n-1}(\res_{n-1}(X))^c}\ar[d]^{\cosk_{n-1}(\phi)^c} \\
      {\sk_{n-1}(\res_{n-1}(Y))^c}\ar[r]_-{(\counit^s_{n-1})^{Y,c}} &
      {Y^c}\ar[r]_-{(\unit^c_{n-1})^{Y,c}} &
      {\cosk_{n-1}(\res_{n-1}(Y))^c}
    }
    \end{equation*}
    This follows by applying lemma \ref{lem:inductive.defn} to the $\scat{C}$-shaped diagram $c \mapsto \phi^c$ taking values in the arrow category $\lcat{M}^\cattwo$.
       \end{obs}

This observation lies at the very heart of the application of the theory of Reedy categories, wherein it is used to construct factorisations, lifts, and so forth via a process of iterated extension from one degree to the next.
 
\begin{ex}\label{ex:sequence-induction}
Consider a diagram $X$ indexed over the poset $\omega$. The colimit defining $\sk_{n-1}X^n$ has a terminal object, from which we see that $\sk_{n-1}X^n = X^{n-1}$. By contrast, $\cosk_{n-1}X^n$ is the terminal object because the hom-sets inside the end \eqref{eq:4} are empty. Hence, observation \ref{obs:inductive.defn} specialises to the obvious fact that to extend a natural transformation between diagrams $X$ and $Y$ indexed by $\omega_{\leq n-1}$ to diagrams indexed by $\omega_{\leq n}$ one must choose objects $X^n$ and $Y^n$ as well as the dotted arrow maps 
\[\xymatrix{ X^0 \ar[d]_{\phi^0} \ar[r] & X^1 \ar[r] \ar[d]_{\phi^1} & X^2 \ar[d]_{\phi^2}\ar[r] & \cdots \ar[r]  & X^{n-1} \ar@{-->}[r] \ar[d]_{\phi^{n-1}} & X^n \ar@{-->}[d]^{\phi^n} \\ Y^0 \ar[r] & Y^1 \ar[r] & Y^2 \ar[r] & \cdots \ar[r]  & Y^{n-1} \ar@{-->}[r] & Y^n}\] 
\end{ex}

\begin{defn}[latching and matching objects]\label{defn:latching}
    Given the centrality of the objects $\sk_{n-1}(\res_{n-1}(X))^c$ and $\cosk_{n-1}(\res_{n-1}(X))^c$ with $n=\deg(c)$ in such extension arguments, these have been dubbed the {\em latching\/} and {\em matching objects\/}, respectively, of $X$ at $c\in\obj(\scat{C})$. Furthermore, the associated functors $\sk_{n-1}(\res_{n-1}({-}))^c$ and $\cosk_{n-1}(\res_{n-1}({-}))^c$ often go by the abbreviated names $\latch^c$ and $\match^c$ (respectively). The associated counit $l^{X,c}\colon\latch^c X\to X^c$ and unit $m^{X,c}\colon X^c\to\match^c X$ maps are called the {\em latching\/} and {\em matching maps\/}, respectively, of $X$ at the object $c\in\scat{C}$.
\end{defn} 

\begin{ex}\label{ex:simp-latching-1}
 Let $X$ be a simplicial set. Examining the defining formulae with great persistence, we see that the $n^{\text{th}}$ latching object $L_nX$ is the set of degenerate $n$-simplices and the latching map $L_nX \to X_n$ is the natural inclusion. Similarly, we intuit that the matching object $M_nX$ is somehow the set of boundary data built from $(n-1)$ and $(n-2)$-simplices for a hypothetical $n$-simplex and that the matching map $X_n \to M_nX$ sends an existing $n$-simplex to the collection of lower-dimensional simplices that define its boundary.
\end{ex}

These computations are simplified once we recast latching and matching objects as weighted colimits and weighted limits, respectively. We will revisit the situation of example \ref{ex:simp-latching-1} in \ref{ex:simp-latching-2}, by which point we will be able to reach the same conclusions in a much more satisfactory manner.
  
\begin{obs}[weights for latching and matching objects]\label{obs:weights-latching}

The latching object of the Yoneda embedding $\scat{C}^\bullet\colon\scat{C}\to\Set^{\scat{C}\op}$ (resp.\ $\scat{C}_\bullet\colon\scat{C}\op\to\Set^{\scat{C}}$) at $c\in\scat{C}$ is called the {\em boundary\/} of the representable $\scat{C}^c$ (resp.\ $\scat{C}_c$) and it is  denoted $\boundary\scat{C}^c$ (resp.\ $\boundary\scat{C}_c$). 
   
Explicitly, fixing an object $c \in \scat{C}$ with degree $n$, $\boundary\scat{C}^c = \latch^c\scat{C}^\bullet= \sk_{n-1}(\scat{C}^c) \in \Set^{\scat{C}}$ and $\boundary\scat{C}_c = \latch_c\scat{C}_\bullet = \sk_{n-1}(\scat{C}_c)$. The counit of the adjunction between left Kan extension and restriction gives rise to the latching maps, which take the form of canonical inclusions $\boundary\scat{C}^c \inc \scat{C}^c$ and $\boundary\scat{C}_c \inc \scat{C}_c$ that we will describe shortly.
    
    Using this notation and the isomorphisms of~\eqref{eq:6}, we obtain the following expressions for the latching and matching objects of an object $X$ of $\lcat{M}^{\scat{C}}$:
    \begin{equation}\label{eq:7}
    \latch^c X \cong \boundary\scat{C}^c\wcolim_{\scat{C}} X \mkern40mu
    \match^c X \cong \wlim{\boundary\scat{C}_c}{X}_{\scat{C}}
    \end{equation}
    The latching and matching maps $\latch^cX \to X^c$ and $X^c \to \match^cX$ are the maps between the weighted limits and weighted colimits \eqref{eq:7} induced from the canonical inclusions $\boundary\scat{C}^c \inc \scat{C}^c$ and $\boundary\scat{C}_c \inc \scat{C}_c$.
  \end{obs}

  Now observation~\ref{obs:skel-reps} provides us with the following concrete description of the skeleta, and hence the boundary, of the representables $\scat{C}^c$ and $\scat{C}_{\bar{c}}$. 

  \begin{lem}[skeleta of the representables of a Reedy category]\label{lem:sk-rep-reedy}
  For each $n\in\mathbb{N}$, the induced map $\sk_n(\scat{C}^c)\to\scat{C}^c$ (resp.\ $\sk_n(\scat{C}_{\bar{c}})\to\scat{C}_{\bar{c}}$) is a monomorphism under which we may identify $\sk_n(\scat{C}^c)$ (resp.\ $\sk_n(\scat{C}_{\bar{c}})$) with the sub-presheaf of $\scat{C}^c$ (resp.\ $\scat{C}_{\bar{c}}$) consisting of those maps $f\colon \bar{c}\to c$ in $\scat{C}$ that factorise through some object of degree at most $n$. More specifically, this latter condition holds if and only if the canonical Reedy factorisation $(\inverse{f},\direct{f})$ factorises $f$ through an object of degree at most $n$. 
  \end{lem}

  \begin{proof}
  This follows immediately from lemma \ref{lem:reedy-fact-connect} and observation \ref{obs:skel-reps}.
  \end{proof}
   	
	\begin{obs}[characterising the boundary of a representable]\label{obs:cons-skeleta}
By lemma \ref{lem:sk-rep-reedy}, we know that a map $f\colon\bar{c}\to c$ is in $\sk_n(\scat{C}^c)$ if and only if the degree of its Reedy factorisation is less than or equal to $n$. Consequently, the map $f$ is in $\sk_n(\scat{C}^c)$ and is not in $\sk_{n-1}(\scat{C}^c)$ if and only if the degree of its Reedy factorisation is actually equal to $n$. In particular, applying these observations to the case $n=\deg(c)$, it follows that $f$ is in $\scat{C}^c$ and is not in $\boundary\scat{C}^c$ if and only if its Reedy factorisation is $(f,\id_c)$ which in turn happens if and only if it is a member of $\inverse{\scat{C}}$. 

Dually, $\boundary\scat{C}_c \subset\scat{C}_c$ is the sub-presheaf of arrows with domain $c$ that do not lie in $\direct{\scat{C}}$.
	\end{obs}

In summary, we have seen that the latching object $\latch^cX$ may be computed as the colimit of $X \in \lcat{M}^{\scat{C}}$ weighted by the sub-presheaf of the representable $\scat{C}^c$ consisting of all arrows with codomain $c$ except for those in $\inverse{\scat{C}}$. Dually, the matching object $\match^cX$ is the limit of $X$ weighted by the sub-presheaf of $\scat{C}_c$ consisting of all arrows with domain $c$ except for those in $\direct{\scat{C}}$. 

\begin{ex}\label{ex:sequence.latching}
Let $X \in \lcat{M}^\omega$, where the poset $\omega$ is given the Reedy category structure described in example \ref{ex:poset.reedy}. We have an isomorphism $\boundary\omega^n\cong\omega^{n-1}$ while $\boundary\omega_n$ is empty. Hence, by the Yoneda lemma \ref{ex:weighted-yoneda}, the $n^{\text{th}}$ latching object is $\latch^nX = X^{n-1}$ and the $n^{\text{th}}$ latching map is the arrow $X^{n-1} \to X^n$ given in the diagram $X$, while the $n^{\text{th}}$ matching object is terminal and the $n^{\text{th}}$ matching map is the unique map $m^n \colon X^n \to *$. Happily, this accords with the observations made in example \ref{ex:sequence-induction}.
\end{ex}

\begin{ex}\label{ex:pushout.latching}
Let $X$ be a diagram of shape $b \leftarrow a \rightarrow c$ in $\lcat{M}$. The latching and matching objects depend on the Reedy structure assigned to the indexing category. Using the Reedy structure of example \ref{ex:poset.reedy}, we compute that $\latch^bX = \latch^c X = X^a$ and that $\latch^aX$ is the initial object $\emptyset$. The maps $l^b$ and $l^c$ are the maps in the diagram; $l^a$ is the unique map from the initial object to $X^a$. The matching objects $\match^aX$, $\match^bX$, and $\match^cX$ are all terminal, again because the hom-sets indexing the products inside the end are empty. 

By contrast, when $b \leftarrow a \rightarrow c$ is given the Reedy structure of example \ref{ex:pushout.reedy.alt}, we have $\latch^aX = \latch^b X = \emptyset$, as the boundaries of the contravariant representables at $a$ and $b$ are empty,  and $\latch^cX=X^a$, as the boundary of the representable at $c$ is isomorphic to the representable at $a$; again the latching maps are the obvious ones. Similarly, $\match^cX=\match^bX = *$ but now $\match^aX=X^b$ and the matching map $m^a \colon X^a \to \match^aX$ is the map in the diagram.
\end{ex}

\begin{ex}\label{ex:parallel.pair.latching}
Let $X$ be a diagram of shape $a \rightrightarrows b$ in $\lcat{M}$. Using the Reedy structure described in example \ref{ex:parallel.pair.reedy}, we see that the boundary of the contravariant representable at $a$ is empty while the boundary of the contravariant representable at $b$ is isomorphic to the coproduct of two copies of the representable at $a$. By concontinuity of the weighted colimit bifunctor, we deduce that $\latch^aX = \emptyset$ and $\latch^b X = X^a \coprod X^a$ with the legs of the latching map $l^b$ defined to be the images of the parallel pair under $X$. 

The matching objects with respect to this Reedy structure are both terminal, as the boundaries of the relevant representables are empty. More interestingly, when we give $a \rightrightarrows b$ the opposite Reedy structure with $\deg(a)=1$ and $\deg(b)=0$, the matching objects are constructed dually to the latching objects above: $\match^b X=*$ and $m^{X,a} \colon X^a \to \match^a X = X^b \times X^b$.
\end{ex}

\begin{ex}\label{ex:simp-latching-2}
Consider the Reedy category $\Del$. By lemma \ref{lem:sk-rep-reedy} and observation \ref{obs:cons-skeleta}, $\boundary\Delta^n$ is the subpresheaf of $\Delta^n$ consisting of those maps which factorise through $[n-1]$, i.e., the subpresheaf generated by the monomorphisms $[k] \inc [n]$ with $k <n$. Happily, this agrees precisely with the simplicial set that is commonly signified by this notation. Dually, write $\boundary\Delta_n$ for the subfunctor of $\Delta_n \colon \Del \to \Set$ whose value at $[k]$ is the set of maps $[n] \to [k]$ in $\Del$ that are not monomorphisms.

Now let $X$ be a simplicial set.  From the weighted limit formula for matching objects \eqref{eq:7} and example \ref{ex:matching.preview}, we see that an element of the $n$-matching object $M_nX$ is a map $\boundary\Delta^n \to X$, i.e., a set of ``boundary data'' in $X$. The matching map $m_n \colon X_n \to M_nX$ sends an $n$-simplex to its boundary. To identify the $n^{\text{th}}$ latching object from the weighted colimit formula \eqref{eq:7}, we make use of the following observation.
\end{ex}

\begin{obs}[latching objects as ordinary colimits]\label{obs:latching.ordinary.colimit}
Observation \ref{obs:weighted.as.ordinary} can be used to express latching and matching objects as ordinary colimits and limits. The category $\el \boundary\scat{C}^c$ is the full subcategory of $\el \scat{C}^c \cong \scat{C}/c$ whose objects are not in $\inverse{\scat{C}}$. As in the proof of lemma \ref{lem:reedy-fact-connect}, the existence and uniqueness of Reedy factorisations implies that this category has a final subcategory which we denote by $\boundary (\direct{\scat{C}}/c)$: it is the full subcategory of the slice category $\direct{\scat{C}}/c$ containing all objects except for the terminal object $\id_c$. 

Given $X \in \lcat{M}^\scat{C}$, $\latch^c X \cong \colim(\boundary (\direct{\scat{C}}/c) \to \scat{C} \xrightarrow{X} \lcat{M})$.  Dually, $\match^c X \cong \lim ( \partial(c/\inverse{\scat{C}}) \to \scat{C} \xrightarrow{X} \lcat{M})$, where $\partial(c/\inverse{\scat{C}})$ is the category whose objects are arrows with domain $c$ that strictly lower degrees and whose morphisms are commuting maps in $\inverse{\scat{C}}$.
\end{obs}

\begin{ex}
Let $X$ be a simplicial set. We have \[\latch^nX \cong \boundary\Del_n \wcolim_{\Del\op} X \cong \colim(\boundary(\inverse{\Del}/[n])\op \to \Del\op \xrightarrow{X} \Set),\] from which we see that $\latch^nX$ is a quotient of the coproduct of $k < n$ of the set of $k$-simplices of $X$ paired with an epimorphism $[n]\twoheadrightarrow[k]$. The quotienting encoded by the coend formula for the weighted colimit identifies those $k$-simplices in the image of a degeneracy map $[k]\twoheadrightarrow[m]$ with the $m$-simplex preimage paired with the composite epimorphism. By the Eilenberg-Zilber lemma, any degenerate $n$-simplex is uniquely expressible as the image of a non-degenerate simplex acted on by a degeneracy map \cite{GabrielZisman:1967:CFHT}. From this, we deduce that the $n^{\text{th}}$ latching object is the subset of degenerate $n$-simplices and the $n^{\text{th}}$ latching map is the canonical inclusion. 
\end{ex}

  \section{Leibniz constructions and the Reedy model structure}\label{sec:Leibniz-Reedy}
  
  If $\scat{C}$ is a Reedy category and $\lcat{M}$ is any model category, then the functor category $\lcat{M}^{\scat{C}}$ admits a \emph{Reedy model structure} in which the weak equivalences are the pointwise weak equivalences. In sections \ref{sec:hocolim} and \ref{sec:connected-weights}, we will see that Reedy model structures are particularly suitable for defining homotopy limits and colimits: The constant diagram functor $\lcat{M} \to \lcat{M}^\scat{C}$ always preserves weak equivalences. Depending on the diagram shape and the chosen Reedy category structure, it is frequently the case that the constant diagram functor is left or right Quillen with respect to the Reedy model structure for \textbf{any} model category $\lcat{M}$. When this is the case, the limit or colimit functor will preserve pointwise weak equivalences between Reedy fibrant or Reedy cofibrant diagrams, respectively.
  
  The (dual) definitions of the cofibrations and fibrations in the Reedy model structures make use of the notions of \emph{relative latching} and \emph{relative matching} maps, which are in turn defined using the \emph{Leibniz construction}, a subject to which we now turn.

\subsection*{\S\ Leibniz constructions}

	\begin{ntn}[exterior products]
		If $\scat{C}$ and $\scat{D}$ are any two small categories and $\lcat{M}$ is a category which possesses all finite products then it will be of some utility to define an {\em exterior product\/} bi-functor $\etimes\colon\lcat{M}^{\scat{C}}\times\lcat{M}^{\scat{D}} \longrightarrow\lcat{M}^{\scat{C}\times\scat{D}}$ which simply carries a pair of functors $X\in\lcat{M}^{\scat{C}}$ and $Y\in\lcat{M}^{\scat{D}}$ to the functor in $\lcat{M}^{\scat{C}\times\scat{D}}$ given by $(X\etimes Y)^{\bar{c},\bar{d}} \defeq X^{\bar{c}}\times Y^{\bar{d}}$. 
	\end{ntn}
	
	For example,  any representable $(\scat{C}\times\scat{D})_{c,d}$ in $\Set^{\scat{C}\times\scat{D}}$ is equal to the exterior product $\scat{C}_c\etimes\scat{D}_d$ of the corresponding representables in $\Set^{\scat{C}}$ and $\Set^{\scat{D}}$. 	
	
	\begin{obs}[Leibniz's formula]\label{obs:box-product}
	Furthermore if $\scat{C}$ and $\scat{D}$ are Reedy categories then we know from observation~\ref{obs:cons-skeleta} that an element $(f,g)$ of $(\scat{C}\times\scat{D})_{c,d}$ is in $\boundary(\scat{C}\times\scat{D})_{c,d}$ if and only if it is not a member of $\direct{\scat{C}\times\scat{D}} = \direct{\scat{C}}\times\direct{\scat{D}}$. This happens if and only if either $f$ is not in $\direct{\scat{C}}$, and is thus an element of $\boundary{\scat{C}_c}$, or $g$ is not in $\direct{\scat{D}}$, and is thus an element of $\boundary{\scat{D}}$. In other words, we see that $\boundary(\scat{C}\times\scat{D})_{c,d}$ is equal to the union $(\boundary\scat{C}_c\etimes\scat{D}_d)\cup(\scat{C}_c\etimes\boundary\scat{D}_d)$. This is simply {\em Leibniz's formula\/} for the boundary of a product of two polygons, or indeed his formula for the derivative of a product of functions!
	\end{obs}

	We may place this observation in the following much more general context:

  \begin{ntn}[arrow categories]
  We use the notation $\cattwo$ to denote the {\em generic arrow}, i.e., the category which has two objects $0$ and $1$ and a single non-identity arrow $0\to 1$. If $\lcat{M}$ is a category, then the functor category $\lcat{M}\mapcat$ is known as its {\em arrow category}. 

  The objects of $\lcat{M}\mapcat$ are in bijective correspondence with the maps of $\lcat{M}$; we shall generally identify these notions. A map of $\lcat{M}\mapcat$ from $f$ to $g$ consists of a pair $(u,v)$ of maps of $\lcat{M}$ which fit into a commutative square:
  \begin{equation*}
    \xymatrix@=2em{
    {A} \ar@{.>}[d]_f \ar[r]^u & {C} \ar@{.>}[d]^g \\
    {B} \ar[r]_v & {D}
    }
  \end{equation*}
  Notice here that we adopt the diagrammatic convention of using dotted arrows to denote those maps that we are regarding as being objects of $\lcat{M}\mapcat$.
  \end{ntn}

  \begin{defn}[the Leibniz construction]\label{defn:leibniz}
  Given a bifunctor
$ \otimes \colon \lcat{K} \times \lcat{L} \to \lcat{M}$ 
  whose codomain  possesses all pushouts, then the {\em Leibniz construction\/} provides us with a bifunctor $ \leib\otimes \colon \lcat{K}\mapcat \times \lcat{L}\mapcat \to \lcat{M}\mapcat$
  between arrow categories which carries a pair of objects $f \in \lcat{K}\mapcat$ and $g \in \lcat{L}\mapcat$ to an object $f\leib\otimes g \in \lcat{M}\mapcat$ which corresponds to the map induced by the universal property of the pushout in the following diagram:
  \begin{equation}\label{eq:13}
  \xymatrix@=2em{
    {K\otimes L} \ar[d]_{K\otimes g} \ar[r]^{f\otimes L} &
    {K'\otimes L} \ar[d] \ar@/^4ex/[ddr]^{K'\otimes g} & \\
    {K\otimes L'} \ar[r] \ar@/_4ex/[rrd]_{f\otimes L'} &
    {(K'\otimes L) \cup_{K\otimes L} (K\otimes L') } \poexcursion
    \ar@{-->}[dr]_{f\leib\otimes g} & \\
    & & {K'\otimes L'}
  }
  \end{equation}
  The action of this functor on the arrows of $\lcat{K}\mapcat$ and $\lcat{L}\mapcat$ is the canonical one induced by the functoriality of $\otimes$ and the universal property of the pushout in the diagram above.
  \end{defn}
	
Particularly in the case where the bifunctor $\otimes$ defines a monoidal product, the Leibniz bifunctor is frequently called the \emph{pushout product}.
	
	\begin{ex}\label{ex:boundary-prod}
		We may apply the Leibniz construction to the exterior product and recast the result of observation~\ref{obs:box-product}  regarding boundaries of representables to say that the boundary inclusion $\boundary(\scat{C}\times\scat{D})_{c,d}\inc(\scat{C}\times\scat{D})_{c,d}$ is canonically isomorphic to the exterior Leibniz product $(\boundary\scat{C}_c\inc\scat{C}_c)\leib\etimes (\boundary\scat{D}_d\inc\scat{D}_d)$.
	\end{ex}

	\begin{obs}[Leibniz and structural isomorphisms]\label{obs:leibniz-isos}
  It is common to work in categories equipped with a range of different bifunctors related by various canonical natural isomorphisms. It is a general fact that when we pass to the corresponding Leibniz bifunctors, we may also construct corresponding natural isomorphisms relating these in a similar fashion, provided that the original bifunctors preserve pushouts in both variables.

  To illustrate this process, suppose we are given a pair of pushout-preserving bifunctors $\tns \colon \lcat{K} \times \lcat{L} \to \lcat{L}$ and $\otimes \colon \lcat{L} \times \lcat{M} \to \lcat{L}$ together with a natural isomorphism \[(K \tns L) \otimes M \cong K \tns (L \otimes M).\] Then it follows from the naturality of these isomorphisms, the definition \ref{defn:leibniz}, and the commutativity of $\tns$ and $\otimes$ with pushouts that these structural isomorphism extend to the Leibniz tensors to give isomorphisms \[(f \leib\tns g) \leib\otimes h \cong f\leib\tns(g\leib\otimes h)\] which are natural in $f \in \lcat{K}\mapcat$, $g \in \lcat{L}\mapcat$, and $h \in \lcat{M}\mapcat$. For example, it follows that if $\otimes$ defines a monoidal structure on a category $\lcat{M}$ with pushouts, then $\leib\otimes$ defines a monoidal structure on $\lcat{M}\mapcat$, with the identity on the unit object serving as the monoidal unit.
	\end{obs}

  \begin{lem}[Leibniz and colimit preservation]\label{lem:leibniz-cocts}
  Suppose that $\lcat{K}$ (resp.\ $\lcat{L}$) and $\lcat{M}$ are cocomplete and that the bifunctor $\otimes\colon\lcat{K}\times\lcat{L}\to\lcat{M}$ is cocontinuous in its first (resp.\ second) variable, i.e., suppose that for each object $L\in\lcat{L}$ the functor ${-}\otimes L\colon\lcat{K}\to\lcat{M}$ preserves all colimits. Then the Leibniz bifunctor $\leib\otimes\colon\lcat{K}\mapcat\times\lcat{L}\mapcat\to\lcat{M}\mapcat$ is also cocontinuous in its first (resp.\ second) variable.
\end{lem}
\begin{proof}
Colimits in the arrow categories $\lcat{K}\mapcat$ and $\lcat{M}\mapcat$ are computed pointwise in $\lcat{K}$ and $\lcat{M}$, respectively. Hence, the proof is completed by the observation that the defining pushout~\eqref{eq:13} commutes with the colimit whose preservation we wish to establish.
 \end{proof}

Frequently, a bifunctor $\otimes \colon \lcat{K} \times \lcat{L} \to \lcat{M}$ is cocontinuous in its first variable because it is right closed, meaning that for each $L \in \lcat{L}$ the functor $-\otimes L$ admits a right adjoint $\hom_r(L,-) \colon \lcat{M} \to \lcat{K}$. By a well-known result of MacLane \cite[IV.7.3]{Maclane:1971:CWM} these right closures assemble into a unique bifunctor $\hom_r \colon \lcat{L}\op \times\lcat{M} \to\lcat{K}$ so that there exist isomorphisms \begin{equation}\label{eq:2varadj} \lcat{M}(K \otimes L,M) \cong \lcat{K}(K,\hom_r(L,M))\end{equation} natural in $K \in \lcat{K}$, $L \in \lcat{L}$, and $M \in \lcat{M}$. 
  
  Given a bifunctor such as $\hom_r$ that is contravariant in one of its variables, it is conventional to apply the Leibniz construction \ref{defn:leibniz} to the opposite functor $\hom_r \colon \lcat{L}\times\lcat{M}\op \to \lcat{K}\op$. Assuming $\lcat{K}$ has pullbacks, the result of this construction applied to maps $g \colon L \to L'$ and $h \colon M \to M'$ yields a diagram in $\lcat{K}$: \[\xymatrix@=2em{ \hom_r(L',M) \ar@/^4ex/[drr]^{\hom_r(g,M)} \ar@/_4ex/[ddr]_{\hom_r(L',h)} \ar@{-->}[dr]^{\wleib{\hom}_r(g,h)} \\ & \hom_r(L',M') \times_{\hom_r(L,M')} \hom_r(L,M) \pbexcursion \ar[d] \ar[r] & \hom_r(L,M) \ar[d]^{\hom_r(L,h)} \\ & \hom_r(L',M') \ar[r]_{\hom_r(g,M')} & \hom_r(L,M')}\] 
  
\begin{lem}[Leibniz and closures]\label{lem:leibniz-close}
The isomorphisms \eqref{eq:2varadj} induce isomorphisms \[ \lcat{M}\mapcat(f \leib\otimes g, h) \cong \lcat{K}\mapcat(f,\wleib{\hom}_r(g,h))\] natural in $f \in \lcat{K}\mapcat$, $g \in \lcat{L}\mapcat$, and $h \in \lcat{M}\mapcat$. In particular, for each $g \in \lcat{L}\mapcat$ there is an adjunction \[ 
\adjdisplay -\leib\otimes g -| \wleib{\hom}_r(g,-) : \lcat{M}\mapcat -> \lcat{K}\mapcat .\]
   \end{lem}
   \begin{proof} A straightforward verification left to the reader.
  \end{proof}
  
 \begin{obs}[Leibniz and lifting properties]\label{obs:leibniz-lifting-properties} As an immediate corollary of lemma \ref{lem:leibniz-close}, observe that there exists a lift in the left-hand diagram if and only if there is also a lift in the right-hand one.
 \[ \xymatrix{ (K' \otimes L) \cup_{K \otimes L} (K \otimes L') \ar[d]_{f \leib\otimes g} \ar[r]^-{(x, y)} & M \ar[d]^h  & K \ar[d]_f \ar[r]^-{\bar{y}} & \hom_r(L',M) \ar[d]^{\wleib{\hom}_r(g,h)} \\ K' \otimes L' \ar@{-->}[ur] \ar[r]_z & M' & K' \ar@{-->}[ur] \ar[r]_-{(\bar{z},\bar{x})} & \hom_r(L',M') \times_{\hom_r(L,M')} \hom_r(L,M)}\] Here the horizontal maps and the dotted arrow lifts are transposes with respect to the adjunctions $-\otimes L \dashv \hom_r(L,-)$ and $-\otimes L' \dashv \hom_r(L',-)$.
 \end{obs} 
 
 \begin{obs}[Leibniz two-variable adjunctions]
 In the case where the bifunctor $\otimes$ is both left and right closed, the bifunctor $\otimes$, the left closure $\hom_l \colon \lcat{K}\op \times \lcat{M} \to \lcat{L}$, and the right closure $\hom_r$ assemble into a \emph{two-variable adjunction}, i.e., there exist isomorphisms \[ \lcat{M}(K \otimes L,M) \cong \lcat{L}(L, \hom_l(K,M)) \cong \lcat{K}(K,\hom_r(L,M))\] natural in all three variables. It follows from lemma \ref{lem:leibniz-close} and the uniqueness statement in \cite[IV.7.3]{Maclane:1971:CWM} that the Leibniz bifunctors $\leib\otimes$, $\wleib{\hom}_l$, and $\wleib{\hom}_r$ also form a two-variable adjunction between the arrow categories.
 \end{obs}
 
\begin{ex} In the case where the bifunctor is a closed monoidal product or, more generally, the tensor bifunctor for a tensored, cotensored, and enriched category, the corresponding Leibniz two-variable adjunction appears in the definition of a monoidal (resp.~enriched) model category; see \cite[\S 4.2]{Hovey:1999fk}. The most familiar example is the product and internal hom on the category of simplicial sets with respect to which the Quillen model structure is a simplicial model category.
\end{ex}

\subsection*{\S\ The Reedy model structure}

Consider a Reedy category $\scat{C}$ and a model category $\lcat{M}$.

  \begin{defn}[relative latching and matching maps]\label{defn:relative-maps}
  It is of interest to apply the Leibniz construction in a context where the bifunctor in question is the weighted colimit $\wcolim_{\scat{C}}\colon \Set^{\scat{C}\op}\times \lcat{M}^{\scat{C}}\to \lcat{M}$. In this case, if $f\colon X\to Y$ is a map in $\lcat{M}^{\scat{C}}$ then the Leibniz colimit\footnote{The Leibniz weighted colimit, which we denote by $\leib\wcolim_{\scat{C}}$, might be more clearly written as $\widehat{\wcolim_{\scat{C}}}$, but we find this latter notation ugly.} $(\boundary\scat{C}^c\inc\scat{C}^c)\leib\wcolim_{\scat{C}} (f\colon X\to Y)$ is called the {\em relative latching map\/} of $f$ at $c$.
  
  To identify the domain and codomain of this map more explicitly, observe that $\scat{C}^c\wcolim_{\scat{C}} X\cong X^c$ and $\scat{C}^c\wcolim_{\scat{C}} Y\cong Y^c$, by Yoneda's lemma, and that we have $\boundary\scat{C}^c\wcolim_{\scat{C}} X\cong \latch^c X$ and $\boundary\scat{C}^c\wcolim_{\scat{C}} Y\cong \latch^c Y$ by~\eqref{eq:7}, so it follows that:
  \begin{equation*}
    (\scat{C}^c \wcolim_{\scat{c}} X) \cup_{\boundary\scat{C}^c \wcolim_{\scat{c}} X} (\boundary\scat{C}^c \wcolim_{\scat{c}} Y) \cong X^c\cup_{\latch^c X} \latch^c Y
  \end{equation*}
  Consequently we find that the relative latching map $(\boundary\scat{C}^c\inc\scat{C}^c)\leib\wcolim_{\scat{C}} f$ is isomorphic to a map of the form:
  \[  \xymatrix{ X^c\cup_{\latch^c X} \latch^c Y \ar[r]^-{\leib{\latch^c}f} & Y^c.}\] 

  Dually we may define the  relative matching maps using the bifunctor obtained by applying the Leibniz construction to the weighted limit bifunctor. Specifically, the {\em relative matching map\/} of $f$ at $c$ is the Leibniz limit $\leibwlim{\boundary\scat{C}_c\inc\scat{C}_c}{f}_\scat{C}$ which is isomorphic to a map of the form:
  \[\xymatrix{ X^c \ar[r]^-{\leib{\match^c}f} & Y^c \times_{\match^c Y} \match^c X. }\] 
  \end{defn}

\begin{ex}
The latching and matching maps of definition \ref{defn:latching} are special cases of the relative latching and matching maps of definition \ref{defn:relative-maps}. The relative latching map of  $\emptyset\to X$ in $\lcat{M}^{\scat{C}}$ at an object $c\in\scat{C}$ is isomorphic to the latching map $\latch^c X\to X^c$, and the relative matching map of $X \to *$ at $c$ is isomorphic to the matching map $X^c \to \match^c X$.
\end{ex}

\begin{obs}[relative latching and matching maps and lifting problems]\label{obs:relative.lifting}
Consider a lifting problem 
\begin{equation}\label{eq:i-lift-p} \xymatrix{ U \ar[d]_i \ar[r]^u & E \ar[d]^p \\ V \ar[r]_v \ar@{-->}[ur]_t & B}\end{equation} 
between maps $i$ and $p$ in $\lcat{M}^\scat{C}$, and suppose that the relative latching maps of $i$ lift against the relative matching maps of $p$.  We might try to construct the components of the lift $t$ inductively by degree: If $c$ has degree zero, then $\leib{\latch^c} i = i^c$ and $\leib{\match^c}p = p^c$. By hypothesis, the maps $i^c$ and $p^c$ lift against each other.

Suppose now that we have constructed a lift $t^d$ for all objects with $\deg(d) < \deg(c)=n$. By observation \ref{obs:inductive.defn}, to define $t^c \colon V^c \to E^c$ it suffices to define any map so that the diagrams
\begin{equation*} \xymatrix{ \latch^c V \ar[d]_{\latch^c(t)} \ar[r]^{l^{V,c}} & V^c \ar@{-->}[d]^{t^c} \ar[r]^{m^{V,c}} & \match^cV \ar[d]^{\match^c(t)} & &U^c \ar[d]_{i^c} \ar[r]^{u^c} & E^c \ar[d]^{p^c} \\ \latch^c E \ar[r]_{l^{E,c}} & E^c \ar[r]_{m^{E,c}} & \match^cE && V^c \ar[r]_{v^c} \ar@{-->}[ur]_{t^c} & B^c}\end{equation*}
commute. We may satisfy both conditions simultaneously by choosing $t^c$ to be a lift in the diagram \[ \xymatrix@C=40pt{ U^c \cup_{\latch^cU} \latch^cV \ar[d]_{\leib{\latch^c}i} \ar[r]^-{(u^c,l^{E,c}\latch^c(t))}& E^c \ar[d]^{\leib{\match^c}p} \\ V^c \ar@{-->}[ur]_{t^c} \ar[r]_-{(v^c,\match^c(t)m^{V,c})} & B^c \times_{\match^c B} \match^cE}\]  By the hypothesis on the relative latching maps of $i$ and the relative matching maps of $p$, such a lift exists.
\end{obs}

Observation \ref{obs:relative.lifting} motivates the appearance of the relative latching and matching maps in the definition of the Reedy model structure, which we now introduce. This model structure was first described in the special case of the dual Reedy categories $\Del$ and $\Del\op$ in \cite{reedy1973htm}.

  \begin{thm}[the Reedy model structure]\label{thm:model-structure}
  Suppose that $\lcat{M}$ is a model category. Then there exists a model structure on $\lcat{M}^{\scat{C}}$ which has:
  \begin{itemize}
  \item Weak equivalences which are those maps $w\colon X\to Y$ in $\lcat{M}^{\scat{C}}$ which are {\em pointwise weak equivalences\/} in the sense that each of its components $w^c\colon X^c\to Y^c$ is a weak equivalence in $\lcat{M}$,
  \item Fibrations, called {\em Reedy fibrations}, which are those maps $p\colon E\to B$ in $\lcat{M}^{\scat{C}}$ for which the relative matching map $\leib{\match^c}(p)$ is a fibration in $\lcat{M}$ for all objects $c\in\scat{C}$, and
  \item Cofibrations, called {\em Reedy cofibrations}, which are those maps $i\colon U\to V$ in $\lcat{M}^{\scat{C}}$ for which the relative latching map $\leib{\latch^c}(i)$ is a cofibration in $\lcat{M}$ for all objects $c\in\scat{C}$.
  \end{itemize} 
  \end{thm}
  
 The key component in the proof of theorem \ref{thm:model-structure} is the fact that any natural transformation $f \colon X \to Y$ in $\lcat{M}^{\scat{C}}$ can be expressed as a countable composite of pushouts of cells built from its relative latching maps (see proposition \ref{prop:building-up}). The proof of this result is a formal calculation with weights: we show in observation \ref{obs:building-skeleta-rep} that the inclusion $\emptyset \inc \scat{C}$ is a countable composite of pushouts of coproducts of exterior Leibniz products $(\boundary\scat{C}_c\inc\scat{C}_c)\leib\etimes(\boundary\scat{C}^c\inc\scat{C}^c)$. 
  
 These technical results are established in the following two sections, which also provides the foundation for the proofs of several key applications of the Reedy model structure appearing below in sections \ref{sec:hocolim}--\ref{sec:simplicial}.
    
\section{Leibniz constructions and cell complexes}\label{sec:Leibniz-rel-cell-cx}

In this section, we investigate the behavior of the Leibniz construction with respect to composites, transfinite composites, and cell complexes in its domain categories. So that we need not continually restate our hypotheses, let us suppose for the duration of this section that we are given a bifunctor $\otimes \colon \lcat{K} \times \lcat{L} \to \lcat{M}$ between cocomplete categories that is cocontinuous in each variable.  It follows from lemma \ref{lem:leibniz-cocts} that the Leibniz tensor $\leib\otimes$ preserves colimits in both variables. Our goal is to describe formulae for several colimits of interest.

Several of the preliminary results appearing below are true under weaker hypotheses. When this is the case, it is fairly obvious, so we are content to leave these formulations to the reader.

  \begin{obs}[Leibniz and composites]\label{obs:leibniz-comp}
Suppose that $f\colon K\to K'$ and $f'\colon K'\to K''$ are maps of $\lcat{K}$ and that $g\colon L\to L'$ is a map in $\lcat{L}$. Let us investigate the relationship between the Leibniz tensor $(f'\circ f)\leib\otimes g$ and the individual Leibniz tensors $f\leib\otimes g$ and $f'\leib\otimes g$.

  To uncover this relationship, we build the following commutative diagram
  \begin{equation*}
    \xymatrix@R=1.5em@C=2.5em{
    {K\otimes L} \ar[r]^{f\otimes L}\ar[d]_{K\otimes g}\ar@{}[dr]|{\textstyle (A)} &
    {K'\otimes L} \ar[r]^{f'\otimes L}\ar[d]\ar@{}[dr]|{\textstyle (B)} &
    {K''\otimes L}\ar[d] \\
    {K\otimes L'} \ar[r]\ar[dr]_{f\otimes L'} &
    {(K'\otimes L)\cup_{K\otimes L} (K\otimes L')}\poexcursion
    \ar[r]\ar[d]^{f\leib\otimes g}\ar@{}[dr]|{\textstyle (C)} &
    {(K''\otimes L)\cup_{K\otimes L} (K\otimes L')}\poexcursion\ar[d]^k
    \ar@/^7em/[dd]^{(f'\circ f)\leib\otimes g} \\
    & {K'\otimes L'}\ar[r]\ar[dr]_{f'\otimes L'} &
    {(K''\otimes L)\cup_{K'\otimes L} (K'\otimes L')}\poexcursion\ar[d]^{f'\leib\otimes g} \\
    && {K''\otimes L'}
    }
  \end{equation*}
  in which the pushout labelled (A) is that used to define $f\leib\otimes g$, the squares labelled (A) and (B) collectively form the pushout used to define $(f'\circ f)\leib\otimes g$ and the squares labelled (B) and (C) collectively form the pushout used to define $f'\leib\otimes g$. Reading from left to right, the vertical composites from top to bottom in this diagram are simply $K\otimes g$, $K'\otimes g$, and $K''\otimes g$ respectively. This diagram then demonstrates that the Leibniz tensor $(f'\circ f)\leib\otimes g$ can be expressed as a composite of the Leibniz tensor $f'\leib\otimes g$ with the map labelled $k$, which is itself a pushout of the Leibniz tensor $f\leib\otimes g$ along the induced map:
  \begin{equation*}
    (f'\otimes L)\cup_{K\otimes L} (K\otimes L')\colon
    \xy <0pt,0pt>*+{(K'\otimes L)\cup_{K\otimes L} (K\otimes L')}\ar 
      <15em,0pt>*+{(K''\otimes L)\cup_{K\otimes L} (K\otimes L')}
    \endxy
  \end{equation*}
 \end{obs}
 
 \begin{ex}[relative latching maps of composites]
 In particular, suppose $f \colon X \to X'$ and  $g\colon X'\to X''$ are natural transformations in $\lcat{M}^{\scat{C}}$, with $\scat{C}$ a Reedy category and $\lcat{M}$ cocomplete. The relative latching map $\leib{\latch^c}(g\circ f)$ may be expressed as a composite of the relative latching map $\leib{\latch^c}(g)$ with a pushout of the relative latching map $\leib{\latch^c}(f)$.
  \end{ex}

  Under appropriate conditions, we can generalise observation \ref{obs:leibniz-comp} to certain important composites of transfinite sequences of maps:

  \begin{defn}[transfinite composites]\label{defn:transfinite-composites}
  Suppose that $\lcat{K}$ is cocomplete category. If $\alpha$ is a ordinal then an {\em $\alpha$-sequence\/} in $\lcat{K}$ is simply a functor $X\colon\alpha\to\lcat{K}$. This is a {\em $\alpha$-composite\/} if and only if for all limit ordinals $\beta<\alpha$ the induced map $\colim_{i<\beta} X^i \to X^\beta$ is an isomorphism. In such a transfinite composite, we will  use the notation $f^{i,j}\colon X^i\to X^j$ to denote the connecting map obtained by applying the functor $X$ to the map corresponding to a pair $i<j \leq\alpha$.

  The term {\em transfinite sequence\/} (resp. {\em transfinite composite}) in $\lcat{K}$ is simply used to denote an object which is an $\alpha$-sequence (resp.\ $\alpha$-composite) in $\lcat{K}$ for some ordinal $\alpha$. A class $\mclass{I}$ of maps of $\lcat{K}$ is {\em closed under transfinite composites\/} if and only if whenever a transfinite composite $X$ has all of its {\em one-step connecting maps\/} $f^{i,i+1}\colon X^i\to X^{i+1}$ (for $i$ with $i+1 < \alpha$) in $\mclass{I}$ then {\em every\/} one of its connecting maps $f^{i,j}\colon X^i\to X^j$ (for $i$ and $j$ with $i < j < \alpha$) is also in $\mclass{I}$.
  \end{defn}

  \begin{defn}[cell complexes]\label{def:rel-cell}
  Now suppose that $\mclass{I}$ is a set (or class) of maps in the cocomplete category $\lcat{K}$. Let $\cell(\mclass{I})$ denote the smallest class of maps of $\lcat{K}$ which contains $\mclass{I}$ and is closed under transfinite composites and pushouts of $\mclass{I}$ along arbitrary maps. It follows that $\cell(\mclass{I})$ is closed under transfinite composites and pushouts and also under coproducts. Furthermore, this construction is order-preserving and idempotent, in the sense that if $\mclass{I}\subseteq\mclass{J}$ then $\cell(\mclass{I})\subseteq\cell(\mclass{J})$ and that $\cell(\cell(\mclass{I})) = \cell(\mclass{I})$ respectively.

  The maps in $\cell(\mclass{I})$ are called {\em relative $\mclass{I}$-cell complexes} or simply {\em $\mclass{I}$-cell complexes}. An object $K$ of $\lcat{K}$ is an {\em $\mclass{I}$-cell complex\/} if the unique map $\emptyset\to K$ is an $\mclass{I}$-cell complex.

  It is a routine matter to show that a map $f\colon X\to X'$ of $\lcat{K}$ is an $\mclass{I}$-cell complex if and only if there exists some transfinite composite $X\colon (\alpha+1)\to\lcat{K}$ with $f = f^{0,\alpha}$ and in which each one-step connecting map $f^{i,i+1}\colon X^i\to X^{i+1}$ is obtained as a pushout
  \begin{equation}\label{eq:10}
    \xymatrix@=2.5em{
    {U_i}\ar[r]^{f'_i}\ar[d]_{u_i} & {V_i} \ar[d]^{v_i} \\
    {X^i}\ar[r]_-{f^{i,i+1}} & {X^{i+1}}\poexcursion
    }
  \end{equation}
  of some map $f'_i\colon U_i\to V_i$ which can be expressed as a coproduct $f'_i\cong\coprod_k f'_{i,k}$ of maps $f'_{i,k}\colon U_{i,k}\to V_{i,k}$ in $\mclass{I}$. In this situation we say that this information {\em presents\/} $f$ as an $\mclass{I}$-cell complex and we refer to the maps $f'_{i,k}$ as the {\em cells\/} of that presentation.
  \end{defn}

  In the proof of the next lemma we will have use for:

  \begin{ntn}\label{ntn:phi-map}
    Now if $f\colon U\to V$ is a map in $\lcat{K}$ then let $\phi(f)$ denote the arrow $(f,\id_V)$ of $\lcat{K}\mapcat$ with domain $f$ and codomain $\id_V$ given by the trivially commutative square:
  \begin{equation*}
    \xymatrix@=2em{
    {U} \ar@{.>}[d]_f \ar[r]^f & {V} \ar@{.>}[d]^{\id_V} \\
    {V} \ar[r]_{\id_V} & {V}
    }
  \end{equation*}
    Suppose also that $g\colon X\to Y$ is another map, and observe that the Leibniz tensor $\id_V\leib\otimes g$ is isomorphic to $\id_{V\otimes Y}$ and that the map $\phi(f)\leib\otimes g\colon f\leib\otimes g\to \id_V\leib\otimes g$ is isomorphic to the map $\phi(f\leib\otimes g)\colon f\leib\otimes g\to \id_{V\otimes Y}$.
  \end{ntn}

  \begin{lem}[Leibniz bifunctors and cell complexes]\label{lem:leibniz-tcofp}
  Fix two maps $f\colon X\to X'$ in $\lcat{K}$ and $g\colon Y\to Y'$ in $\lcat{L}$ and suppose that we are given a presentation of $f$ as a cell complex with cells $f'_{i,k}$. Then we may present the Leibniz tensor $f\leib\otimes g$ as a cell complex with cells $f'_{i,k}\leib\otimes g$.
  \end{lem}

  \begin{proof}
  To fix notation, suppose that $f$ is presented by a transfinite composite $X\colon (\alpha+1)\to \lcat{K}$ in which each one-step connecting map $f^{i,i+1}$ is displayed as a pushout of the coproduct $f'_i\defeq\coprod_k f'_{i,k}$. Observe that we may construct corresponding gadgets in the arrow category $\lcat{K}\mapcat$. Specifically we may construct an $(\alpha+1)$-composite $X/X^\alpha\colon(\alpha+1)\to\lcat{K}\mapcat$ whose object at $i$ is $f^{i,\alpha}$ and which carries the map $i\leq j$ of $(\alpha+1)$ to the arrow
  \begin{equation*}
    \xymatrix@=3em{
      {X^i}\ar[r]^{f^{i,j}}\ar@{.>}[d]_{f^{i,\alpha}} & 
      {X^j}\ar@{.>}[d]^{f^{j,\alpha}} \\
      {X^\alpha}\ar@{=}[r] & {X^\alpha}
    }
  \end{equation*}
  in $\lcat{K}\mapcat$. Notice here that the connecting map between $i$ and $\alpha$ of this transfinite sequence is simply the map $\phi(f^{i,\alpha})$ introduced in notation~\ref{ntn:phi-map}.

  To check that this is indeed a transfinite composite, we must show that for each limit ordinal $\beta\leq\alpha$ the cocone of maps $(f^{i,\beta},\id_{X^\alpha})\colon f^{i,\alpha}\to f^{\beta,\alpha}$ induces an isomorphism $\colim_{i<\beta}  f^{i,\alpha}\cong f^{\beta,\alpha}$. However this result holds immediately simply because the corresponding property holds for the transfinite composite $X$ in $\lcat{K}$ and colimits of sequences in $\lcat{K}\mapcat$ are constructed pointwise in $\lcat{K}$. Furthermore, each pushout of~\eqref{eq:10} gives rise to a corresponding cube
  \begin{equation}\label{eq:11}
    \xymatrix@R=2em@C=3em{
      {U_i}\ar[rr]^{f'_i}\ar[dd]_{u_i}\ar@{.>}[dr]^{f'_i} &&
      {V_i}\ar'[d]!/u 4pt/[dd]^{v_i}
      \ar@{:}[dr]^{\id_{V_i}} & \\
      & {V_i} \ar@{=}[rr]\ar[dd]_(0.3){f^{i+1,\alpha}\circ v_i}
      && {V_i}\ar[dd]^{f^{i+1,\alpha}\circ v_i} \\
      {X^i}\ar'[r][rr]_-{f^{i,i+1}}\ar@{.>}[dr]_{f^{i,\alpha}} && 
      {X^{i+1}}\poexcursion\ar@{.>}[dr]^(0.4){f^{i+1,\alpha}} & \\
      & {X^\alpha}\ar@{=}[rr] && \poexcursion{X^\alpha} 
    }
  \end{equation}
  which is again a pushout in $\lcat{K}\mapcat$, simply because such things are computed pointwise in $\lcat{K}$. 
  
  Observe now that the upper face of the cube~\eqref{eq:11} is simply the map $\phi(f'_i)$ of $\lcat{K}\mapcat$ as defined in notation~\ref{ntn:phi-map}. Furthermore, expanding $f'_i$ as a coproduct of cells we see that $\phi(f'_i)\cong\phi(\coprod_k f'_{i,k})\cong\coprod_k \phi(f'_{i,k})$. So, summarising all of the information of the last few paragraphs, we have shown that the map $\phi(f)$ of $\lcat{K}\mapcat$ can be presented as a cell complex with cells $\phi(f'_{i,k})$.

  Now we know that the Leibniz functor ${-}\leib\otimes g\colon\lcat{K}\mapcat\to\lcat{M}\mapcat$ is cocontinuous, by lemma~\ref{lem:leibniz-cocts}. So when we apply it to the structures derived in the last few paragraphs it preserves the colimits there and carries that information to a presentation of $\phi(f)\leib\otimes g\cong \phi(f\leib\otimes g)$ (cf.\ notation~\ref{ntn:phi-map}) as a cell complex with cells $\phi(f'_{i,k})\leib\otimes g\cong \phi(f'_{i,k}\leib\otimes g)$. 

  Finally, the domain projection functor $\dom\colon\lcat{M}\mapcat\to\lcat{M}$ also preserves colimits, since these are constructed pointwise in $\lcat{M}$. Consequently it too preserves the colimits involved in the presentation of $\phi(f\leib\otimes g)$ derived in the last paragraph, and so it carries that information to a presentation of $f\leib\otimes g = \dom(\phi(f\leib\otimes g))$ as a cell complex with cells $f'_{i,k}\leib\otimes g = \dom(\phi(f'_{i,k}\leib\otimes g))$ as required.
  \end{proof}

  \begin{obs}\label{obs:leibniz-tcofp}
    Tracing through the argument of lemma \ref{lem:leibniz-tcofp}, we see that the transfinite composite we constructed there, whose composite is $f\leib\otimes g$, carries the index $i\leq \alpha$ to the domain of $f^{i,\alpha}\leib\otimes g$. In other words, this is the object given by the following pushout:
    \begin{equation*}
      \xymatrix@=1.5em{
        {X^i\otimes Y}\ar[r]^{X^i\otimes g}
        \ar[d]_{f^{i,\alpha}\otimes Y} &
        {X^i\otimes Y'}\ar[d] \\
        {X^\alpha\otimes Y}\ar[r] &
        {(X^\alpha\otimes Y)\cup_{X^i\otimes Y}
        (X^i\otimes Y')}\poexcursion
      }
    \end{equation*}
    The connecting map from index $i$ to index $i+1$ in this transfinite sequence is given by a pushout:
    \begin{equation*}
      \xymatrix@R=1.5em@C=2em{
        {\coprod_k (V_{i,k}\otimes Y)
        \cup_{U_{i,k}\otimes Y} (U_{i,k}\otimes Y')}
        \ar[r]^-{\coprod_k f'_{i,k}\leib\otimes g}\ar[d] &
        {\coprod_k V_{i,k}\otimes Y'}\ar[d] \\
        {(X^\alpha\otimes Y)\cup_{X^i\otimes Y}
        (X^i\otimes Y')}\ar[r] &
        {(X^\alpha\otimes Y)\cup_{X^{i+1}\otimes Y}
        (X^{i+1}\otimes Y')}\poexcursion
      }
    \end{equation*}
  \end{obs}

  \begin{ex}[relative latching maps of composites II]\label{ex:latch-comp}
    Suppose that a natural transformation $f\colon X\to X'$ in $\lcat{M}^{\scat{C}}$ admits a presentation as a cell complex with cells $f'_{i,k}\colon U_{i,k}\to V_{i,k}$. Then we may apply the last result to show that the relative latching map $\leib{\latch^c}(f)\cong(\boundary\scat{C}^c\inc\scat{C}^c)\leib\wcolim_{\scat{C}} f$ admits a presentation as a cell complex with cells the relative latching maps $\leib{\latch^c}(f'_{i,k})\cong(\boundary\scat{C}^c\inc\scat{C}^c)\leib\wcolim_{\scat{C}} f'_{i,k}$.
  \end{ex}
	
	In summary, the Leibniz tensor $-\leib\otimes g$ preserves cell structures. It is now straightforward to extend this result to cell complexes in both variables.

	\begin{cor}\label{cor:leibniz-tcofp}
  Suppose that $f\colon X\to X'$ and $g\colon Y\to Y'$ admit presentations as cell complexes with cells $f'_{i,k}\colon U_{i,k}\to V_{i,k}$ and $g'_{j,l}\colon A_{j,l}\to B_{j,l}$ respectively. Then the Leibniz tensor $f\leib\otimes g$ admits a presentation as a cell complex with cells $f'_{i,k}\leib\otimes g'_{j,l}$.
	\end{cor}
	
	\begin{proof}
		Simply apply the result in the last lemma to first show that $f\leib\otimes g$ admits a presentation as a cell complex with cells $f'_{i,k}\leib\otimes g$. Now apply that same result again to each of these latter Leibniz tensors to show that each $f'_{i,k}\leib\otimes g$ admits a presentation as a cell complex with cells $f'_{i,k}\leib\otimes g'_{j,l}$. Finally, observe that pushouts of transfinite composites of pushouts are again transfinite composites of pushouts and that transfinite composites of transfinite composites are transfinite composites. So our result follows.
	\end{proof}

  If $\mclass{I}$ is a set of maps in $\lcat{K}$ and $\mclass{J}$ is a set of maps in $\lcat{L}$ then let $\mclass{I}\leib\otimes\mclass{J}$ denote the set $\{f\leib\otimes g \mid f\in\mclass{I}, g\in\mclass{J}\}$. Now corollary~\ref{cor:leibniz-tcofp} leads immediately to the following proposition, the applications of which are myriad in homotopy theory:

  \begin{prop}\label{prop:leibniz-tcofp}
  Let $\otimes \colon \lcat{K} \times \lcat{L} \to \lcat{M}$ be a cocontinuous bifunctor between cocomplete categories, and let $\mclass{I}$ and $\mclass{J}$ be any sets of maps in $\lcat{K}$ and $\lcat{L}$ respectively. It follows that $\cell(\mclass{I})\leib\otimes\cell(\mclass{J}) \subseteq \cell(\mclass{I}\leib\otimes\mclass{J})$.  
   \end{prop}

  \begin{proof}
  Simply apply  corollary~\ref{cor:leibniz-tcofp} directly to the characterisations of $\cell(\mclass{I})$, $\cell(\mclass{J})$, and $\cell(\mclass{I}\leib\otimes\mclass{J})$ given in definition~\ref{def:rel-cell}.
  \end{proof}

\section{Cellular presentations and Reedy categories}\label{sec:cellular}

    The work of the current section provides an ideal exemplar of a philosophy which lies at the very core of our approach throughout this paper, and which we might summarise in the phrase ``It's all in the weights!'' Under this philosophy we take a two step approach:
    \begin{enumerate}[label=\arabic*.]
      \item first establish a corresponding result for certain set-valued presheaves by straightforward direct computation, then
      \item extend that result to general diagrams indexed by a Reedy category using weighted (co)limits and Leibniz constructions.
    \end{enumerate}
    When combined with the Yoneda lemma, in the form given in example~\ref{ex:weighted-yoneda}, this approach often allows us to reduce results involving objects in general functor categories $\lcat{M}^{\scat{C}}$, where $\scat{C}$ is a Reedy category, to explicit computations involving (sub-objects of) representables.

Our immediate aim is to apply this philosophy to showing that any natural transformation $f\colon X\to Y$ in $\lcat{M}^{\scat{C}}$ can be expressed as a countable composite of pushouts of cells built from its relative latching maps. Our first step towards this result is to provide a combinatorially explicit result of this type for the skeleta of representables:

  \begin{obs}[skeleta of two-sided representables]\label{obs:two-sided}
    It will be convenient to re-express our results about skeleta of representables, as discussed in lemma~\ref{lem:sk-rep-reedy} and observation~\ref{obs:cons-skeleta}, in a more symmetric {\em two-sided\/} form. So start with the two variable hom-functor $\scat{C}$ in $\Set^{\scat{C}\op\times\scat{C}}$ and define $\sk_n(\scat{C})$ to be its subobject of those maps $f\colon\bar{c}\to c$ which factorise through some object of degree less than or equal to $n$. As one might hope, the explicit description furnished by lemma~\ref{lem:sk-rep-reedy} tells us that the skeleta of covariant and contravariant representables may both be captured in terms of these two-sided skeleta, specifically $\sk_n(\scat{C}^c) \cong \sk_n(\scat{C})^c$ and $\sk_n(\scat{C}_{\bar{c}}) \cong \sk_n(\scat{C})_{\bar{c}}$.

When $W$ is an object in $\Set^{\scat{C}\times\scat{D}\op}$ and $X$ is an object in $\lcat{M}^{\scat{D}}$ we shall extend our weighted colimit notation in an obvious fashion and write $W\wcolim_{\scat{D}} X$ to denote the object of $\lcat{M}^{\scat{C}}$ given by $(W\wcolim_{\scat{D}} X)^c \defeq W^c\wcolim_{\scat{D}} X$. 

    Armed with these conventions, we may write the Yoneda lemma, as expressed in example~\ref{ex:weighted-yoneda}, and the formula for the $n$-skeleton of an object $X$ of $\lcat{M}^{\scat{C}}$, as given in lemma~\ref{lem:skeleta-colimit}, in the following particularly simple forms:
    \begin{equation*}
      X\cong\scat{C}\wcolim_{\scat{C}} X \mkern60mu
      \sk_n(X)\cong\sk_n(\scat{C})\wcolim_{\scat{C}} X. 
    \end{equation*}
  \end{obs} 

	\begin{obs}[building up for skeleta of representables]\label{obs:building-skeleta-rep}
		For each object $c\in\scat{C}$ there is a map $\circ\colon\scat{C}_c\etimes\scat{C}^c\to\scat{C}$ which simply carries a pair of maps $f\colon c\to c'$ in $\scat{C}_c$ and $g\colon \bar{c}\to c$ in $\scat{C}^c$ to their composite $f\circ g\colon \bar{c}\to c'$ in $\scat{C}$. When it is the case that $\deg(c)\leq n$ then this composition map factorises through the $n$-skeleton $\sk_n(\scat{C})\inc\scat{C}$ and we may collect such maps together to give an induced map \[\xymatrix{ \left(\displaystyle\coprod_{c\in\scat{C}, \deg(c)=n} \scat{C}_c\etimes\scat{C}^c\right) \ar[r]^-{\circ}  & \sk_n(\scat{C}).} \] 
		
		Now consider a pair $(f,g)$ in some $\scat{C}_c\etimes\scat{C}^c$ with $\deg(c)=n$. If $f\in\boundary\scat{C}_c = \sk_{n-1}(\scat{C}_c)$ or if $g\in\boundary\scat{C}^c = \sk_{n-1}(\scat{C}^c)$ then the map that lies in the boundary  factorises through an object of degree less than $n$. Hence,  the composite $f\circ g$ also factorises through that same object and is thus an element of $\sk_{n-1}(\scat{C})$. Conversely when $f\notin\boundary\scat{C}_c$ and $g\notin\boundary\scat{C}^c$ then observation~\ref{obs:cons-skeleta} tells us that $f\in\direct{\scat{C}}$ and $g\in\inverse{\scat{C}}$ so it follows that $(f,g)$ is the Reedy factorisation of $f\circ g$. However, according to lemma~\ref{lem:reedy-fact-connect} this Reedy factorisation has minimal degree amongst all factorisations of $f\circ g$ and its degree is equal to $\deg(c)=n$, so this composite cannot be in $\sk_{n-1}(\scat{C})$. So in summary we have shown that $(f,g)$ is in $(\boundary\scat{C}_c\etimes\scat{C}^c) \cup (\scat{C}_c\etimes\boundary\scat{C}^c)$ if and only if $f\circ g$ is in $\sk_{n-1}(\scat{C})$. This tells us precisely that our composition map restricts to give a commutative square
		\begin{equation*}
			\xymatrix@=3em{
				{\displaystyle\coprod_{\hbox to 2em{\hss$\mathsurround=0pt\scriptstyle c\in\scat{C}, \deg(c)=n$\hss}} (\boundary\scat{C}_c\etimes\scat{C}^c) \cup (\scat{C}_c\etimes\boundary\scat{C}^c)} \pbexcursion\ar@{u(->}[r]\ar"1,1"-<0ex,2.5ex>;"2,1"_-{\circ} &
				{\displaystyle\coprod_{\hbox to 2em{\hss$\mathsurround=0pt\scriptstyle c\in\scat{C}, \deg(c)=n$\hss}} \scat{C}_c\etimes\scat{C}^c}
			\ar[d]^-{\circ} \\
			{\sk_{n-1}(\scat{C})}\ar@{u(->}[r] & {\sk_n(\scat{C})}}
		\end{equation*}
		which is a pullback in $\Set^{\scat{C}\op\times\scat{C}}$. 
		
		Finally observation~\ref{obs:cons-skeleta} also tells us that a map $h$ is in $\sk_n(\scat{C})$ and is not in $\sk_{n-1}(\scat{C})$ if and only if its unique Reedy factorisation $(\inverse{h},\direct{h})$ factorises it through an object $c$ of degree $n$. It follows therefore that this $c$ is the unique object of degree $n$ for which $h$ appears in the image of $\circ\colon\scat{C}_c\etimes\scat{C}^c\to\sk_n(\scat{C})$. These results are now enough to demonstrate that the square above is also a pushout.

In summary, the skeleta of the two-sided representable $\scat{C}$ define an $\omega$-composite \[\emptyset \inc \sk_0(\scat{C})\inc\sk_1(\scat{C})\inc\cdots \inc \sk_n(\scat{C}) \inc \cdots \colim = \scat{C}.\] What we have done here is to demonstrate that the inclusion $\emptyset\inc\scat{C}$ in $\Set^{\scat{C}\op\times\scat{C}}$ admits a presentation as a cell complex whose cells at the $n^{\text{th}}$ step are the exterior Leibniz products $(\boundary\scat{C}_c\inc\scat{C}_c)\leib\etimes(\boundary\scat{C}^c\inc\scat{C}^c)$ for $c\in\scat{C}$ with $\deg(c)=n$.
	\end{obs}
		
With this inductive presentation of the skeleta of representables under our belts, we can now use lemma~\ref{lem:leibniz-tcofp} to construct a corresponding presentation of the skeleta of an arbitrary object in $\lcat{M}^{\scat{C}}$. In this regard, we might say that the following proposition expresses the key geometric character of the relative latching maps:
	
\begin{prop}[general building up]\label{prop:building-up} Any natural transformation $f \colon X \to Y$ in $\lcat{M}^{\scat{C}}$ admits a presentation as a cell complex whose countable composite is of the form
\begin{equation}\label{eq:skel-seq}
  \xymatrix@C=2em@R=1ex{
    {X}\ar[r] &
    {X\mathop{\cup}\limits_{\sk_0X} \sk_0Y}\ar[r] &
    {X\mathop{\cup}\limits_{\sk_1X} \sk_1Y}\ar[r] &
    {\cdots}\ar[r] &
    {X\mathop{\cup}\limits_{\sk_nX} \sk_nY}\ar[r] &
    {\cdots}
  }
\end{equation}
and for which the cells at the $n^{\text{th}}$ step are the natural transformations
\begin{equation}\label{eq:c-cell}
  (\boundary\scat{C}_c\inc\scat{C}_c)\leib\tns \leib{\latch^c}(f)
\end{equation}
associated with the latching maps of $f$ at objects $c\in\scat{C}$ with $\deg(c)=n$.
\end{prop}

\begin{proof}
By the Yoneda lemma, $f$ is isomorphic to the Leibniz tensor with the inclusion of the empty set into the hom bifunctor, i.e., \[ f \cong (\emptyset \inc \scat{C}) \leib\wcolim_{\scat{C}} f.\] Now we know from observation~\ref{obs:building-skeleta-rep} that the inclusion $\emptyset\inc\scat{C}$ admits a presentation as a cell complex whose cells at $n^{\text{th}}$ step are the exterior Leibniz products $(\boundary\scat{C}_c\inc\scat{C}_c)\leib\etimes(\boundary\scat{C}^c\inc\scat{C}^c)$ for $c\in\scat{C}$ with $\deg(c)=n$. So, applying lemma~\ref{lem:leibniz-tcofp}, we find that $f\cong (\emptyset \inc \scat{C}) \leib\wcolim_{\scat{C}} f$ admits a presentation as a cell complex whose cells at the $n^{\text{th}}$ step are the Leibniz colimits
\begin{equation}\label{eq:skel-seq-cells}
  ((\boundary\scat{C}_c\inc\scat{C}_c)\leib\etimes(\boundary\scat{C}^c\inc\scat{C}^c)) \leib\wcolim_{\scat{C}} f
\end{equation}
for $c\in\scat{C}$ with $\deg(c)=n$.

Now an easy computation verifies that if $U$ is an object of $\Set^{\scat{C}}$, $V$ is an object of $\Set^{\scat{D}\op}$ and $X$ is an object of $\lcat{M}^{\scat{D}}$ then the object $(U\etimes V)\wcolim_{\scat{D}} X$ is naturally isomorphic to $U\tns(V\wcolim_{\scat{D}} X)$. As observed in \ref{obs:leibniz-isos}, these isomorphisms pass to the corresponding Leibniz operations, giving isomorphisms 
\begin{align*}
  ((\boundary\scat{C}_c\inc\scat{C}_c)\leib\etimes(\boundary\scat{C}^c\inc\scat{C}^c))\leib\wcolim_{\scat{C}}f &\cong (\boundary\scat{C}_c\inc\scat{C}_c)\leib\tns((\boundary\scat{C}^c\inc\scat{C}^c)\leib\wcolim_{\scat{C}}f) \\
  &\cong (\boundary\scat{C}_c\inc\scat{C}_c)\leib\tns \leib{\latch^c}(f)
\end{align*}
which reduce the cells of the presentation we've constructed, as displayed in~\eqref{eq:skel-seq-cells}, to the form given in equation~\eqref{eq:c-cell} of the statement.

  Now applying the Yoneda lemma and our formulae for skeleta as expressed in observation~\ref{obs:two-sided}, we find that the Leibniz colimit $(\sk_n(\scat{C})\inc \scat{C})\leib\wcolim_{\scat{C}} f$ is isomorphic to the unique (dashed) map induced by the pushout in the following diagram:
\begin{equation*}
  \xymatrix{
    {\sk_nX}\ar[r]^{\sk_n(f)}\ar[d] &
    {\sk_nY}\ar[d]\ar@/^1ex/[ddr] \\
    {X}\ar[r]\ar@/_1ex/[drr]_f & 
    {X\mathop{\cup}\limits_{\sk_nX} \sk_nY}\poexcursion 
    \ar@{-->}[dr] \\
    && {Y}
  }
\end{equation*}
It follows that when we apply the argument in the proof of lemma~\ref{lem:leibniz-tcofp} to the countable sequence of skeleta $\sk_n(\scat{C})$ of observation~\ref{obs:building-skeleta-rep} we obtain the countable sequence displayed in equation~\eqref{eq:skel-seq} of the statement as required.
\end{proof}

  \begin{cor}\label{cor:building-up}
    Let $\mclass{B}$ denote the set of boundary inclusions $\{\boundary\scat{C}_c\inc\scat{C}_c\mid c\in\scat{C}\}$ in $\Set^{\scat{C}}$ and suppose that $\mclass{I}$ is a class of maps in $\lcat{M}$. Then a map $f$ of $\lcat{M}^{\scat{C}}$ is in $\cell(\mclass{B}\leib\tns\mclass{I})$ if and only if its relative latching maps are all in $\cell(\mclass{I})$.
  \end{cor}

  \begin{proof} This is now a straightforward matter of applying formal manipulations with the Leibniz operation:
    \begin{description} 
    \item[\textbf{``if''}] We know, from proposition~\ref{prop:building-up}, that $f$ admits a presentation as a cell complex whose cells are maps of the form $(\boundary\scat{C}_c\inc\scat{C}_c)\leib\tns \leib{\latch^c}f$. So if $f$ is a map which satisfies the assumption of the statement that each of its latching maps is in $\cell(\mclass{I})$ then the presentation of the last sentence suffices to demonstrate that $f$ is an element of $\cell(\mclass{B}\leib\tns\cell(\mclass{I}))$. Now $\mclass{B}\subseteq \cell(\mclass{B})$ from which we may infer that $\mclass{B}\leib\tns\cell(\mclass{I})\subseteq\cell(\mclass{B})\leib\tns\cell(\mclass{I})$, furthermore proposition~\ref{prop:leibniz-tcofp} demonstrates that $\cell(\mclass{B})\leib\tns\cell(\mclass{I})\subseteq\cell(\mclass{B}\leib\tns\mclass{I})$ so combining these inclusions it follows that $\mclass{B}\leib\tns\cell(\mclass{I})\subseteq\cell(\mclass{B}\leib\tns\mclass{I})$. Taking cell complexes on both sides of this latter inclusion we get $\cell(\mclass{B}\leib\tns\cell(\mclass{I}))\subseteq\cell(\cell(\mclass{B}\leib\tns\mclass{I}))=\cell(\mclass{B}\leib\tns\mclass{I})$ so since we have already shown that $f$ is in $\cell(\mclass{B}\leib\tns\cell(\mclass{I}))$ it follows therefore that it is in $\cell(\mclass{B}\leib\tns\mclass{I})$ as required.

    \item[\textbf{``only if''}] If $f$ is in $\cell(\mclass{B}\leib\tns\mclass{I})$ then it admits a presentation as a cell complex whose cells are of the form $(\boundary\scat{C}_{\bar{c}}\inc\scat{C}_{\bar{c}})\leib\tns f'_i$ with $f'_i\in\mclass{I}$. So, by example~\ref{ex:latch-comp}, we know that the relative latching map $\leib{\latch^c}f \cong (\boundary\scat{C}^c\inc\scat{C}^c)\leib\wcolim_{\scat{C}} f$ admits a presentation as a cell complex whose cells are of the form $(\boundary\scat{C}^c\inc\scat{C}^c)\leib\wcolim_{\scat{C}}((\boundary\scat{C}_{\bar{c}}\inc\scat{C}_{\bar{c}})\leib\tns f'_i)$ with $f'_i\in\mclass{I}$.

    Now if $U$ is an object in $\Set^{\scat{C}\op}$, $V$ is an object of $\Set^{\scat{C}}$ and $X$ is an object of $\lcat{M}$ then it is easily checked that the object $U\wcolim_{\scat{C}} (V\tns X)$ is naturally isomorphic to $(U\wcolim_{\scat{C}} V)\tns X$ in $\lcat{M}$. As observed in~\ref{obs:leibniz-isos}, these isomorphisms pass to the corresponding Leibniz operations, which in particular provide us with an isomorphism:
    \begin{equation*}
      (\boundary\scat{C}^c\inc\scat{C}^c)\leib\wcolim_{\scat{C}}((\boundary\scat{C}_{\bar{c}}\inc\scat{C}_{\bar{c}})\leib\tns f'_i)
      \cong
      ((\boundary\scat{C}^c\inc\scat{C}^c)\leib\wcolim_{\scat{C}}((\boundary\scat{C}_{\bar{c}}\inc\scat{C}_{\bar{c}}))\leib\tns f'_i
    \end{equation*}
    Furthermore a simple computation, using Yoneda's lemma in the form of example~\ref{ex:latch-comp}, reveals that the Leibniz colimit $(\boundary\scat{C}^c\inc\scat{C}^c)\leib\wcolim_{\scat{C}}(\boundary\scat{C}_{\bar{c}}\inc\scat{C}_{\bar{c}})$ is isomorphic to the set inclusion $(\boundary\scat{C}_{\bar{c}})^c\cup(\boundary\scat{C}^c)_{\bar{c}}\inc\scat{C}_{\bar{c}}^c$. So we have succeeded in showing that the relative latching map $\latch^c(f)$ admits a presentation as a cell complex whose cells are of the form $((\boundary\scat{C}_{\bar{c}})^c\cup(\boundary\scat{C}^c)_{\bar{c}}\inc\scat{C}_{\bar{c}}^c)\leib\tns f'_i$ with $f'_i\in\mclass{I}$.
  
  Now observe that if $i\colon U\inc V$ is an inclusion of sets and $g\colon A\to B$ is any map in $\lcat{M}$ then the Leibniz tensor $(i\colon U\inc V)\leib\tns (g\colon A\to B)$ is isomorphic to a coproduct $((V\setminus U)\tns g) \cprod (U\tns B)$. In particular, it follows that each $((\boundary\scat{C}_{\bar{c}})^c\cup(\boundary\scat{C}^c)_{\bar{c}}\inc\scat{C}_{\bar{c}}^c)\leib\tns f'_i$ may be expressed as a coproduct of a certain number of copies of $f'_i$ and a certain number of copies of the identity on the codomain of $f'_i$. However $\cell(\mclass{I})$ is closed under coproducts and identities, so since $f'_i$ is in there so is the coproduct of the last sentence. Finally, we have now shown that $\leib{\latch^c}f$ admits a presentation as a cell complex all of whose cells are in $\cell(\mclass{I})$, from which we may infer that $\leib{\latch^c}f$ is a member of $\cell(\cell(\mclass{I}))=\cell(\mclass{I})$ as required.\qedhere
    \end{description}
  \end{proof}

  \begin{cor}\label{cor:B-cell-complex}
    A map $f\colon X\to Y$ in $\Set^{\scat{C}}$ is a $\mclass{B}$-cell complex if and only if its relative latching maps are all monomorphisms.
  \end{cor}
  
  \begin{proof}
    Immediate from the last proposition by simply taking $\mclass{I}$ to be the set $\{\emptyset\inc 1\}$ for which $\cell(\mclass{I})$ is the class of injective maps in $\Set$.
  \end{proof}

  \begin{ex}\label{ex:EZ.simp.set}
    The Eilenberg-Zilber lemma implies that any monomorphism $m\colon X\to Y$ in the category of simplicial sets has monomorphic relative latching maps \cite{GabrielZisman:1967:CFHT}. So on applying corollary \ref{cor:B-cell-complex} we recover the usual skeletal decomposition of $m$ and the fact that we can build up $m$ by successively adjoining standard simplices along their boundaries.

More generally, the conditions of corollary \ref{cor:B-cell-complex} are satisfied whenever $\scat{C}\op$ is an \emph{elegant Reedy category} in the sense of \cite{BergnerRezk:2012rc}.
  \end{ex}
  
\section{Proof of the Reedy model structure}\label{sec:proof}

  Suppose now that $\lcat{M}$ is a model category and $\scat{C}$ is a Reedy category. As an example of the utility of our presentation of the theory of Reedy categories, we now present a relatively efficient proof of theorem \ref{thm:model-structure}, establishing the Reedy model structure on  $\lcat{M}^{\scat{C}}$.

  \begin{lem}\label{lem:triv-cof-char}
A map $f \in \lcat{M}^{\scat{C}}$ is both a Reedy cofibration and a pointwise weak equivalence if and only if the relative latching map $\leib{\latch^c}f$ is a trivial cofibration in $\lcat{M}$ for all objects $c\in\scat{C}$.
  \end{lem}

In other words, a map is a Reedy trivial cofibration just when each of its relative latching maps is a trivial cofibration.

  \begin{proof}
  Suppose $f$ is a Reedy cofibration. We show that $f$ is a pointwise weak equivalence if and only if each relative latching map is a weak equivalence. Note that if $c$ has degree zero, the relative latching map $\leib{\latch^c}f$ is simply the map $f^c$, so these conditions coincide. We proceed inductively by considering an object $c\in\scat{C}$ and assuming that the relative latching map $\leib{\latch^d}f$ and the component $f^d$ are both weak equivalences whenever $d\in\scat{C}$ is an object with $\deg(d) < \deg(c)$. 

  By Yoneda's lemma we know that the Leibniz colimit $(\emptyset\inc\scat{C}^c)\leib\wcolim_{\scat{C}} f\cong \scat{C}^c\wcolim_{\scat{C}}f$ is simply isomorphic to the component $f^c$ of our map $f$. Furthermore we may decompose the inclusion $\emptyset\inc\scat{C}^c$ as a composite of $\emptyset\inc\boundary\scat{C}^c$ and $\boundary\scat{C}^c\inc\scat{C}^c$ thus, applying observation~\ref{obs:leibniz-comp}, we see that $f^c$ is isomorphic to a composite of the relative latching map $\leib{\latch^c}f$ and a pushout of the Leibniz tensor $(\emptyset\inc\boundary\scat{C}^c)\leib\wcolim_{\scat{C}}f \cong \boundary\scat{C}^c \wcolim_{\scat{C}}f	$. 

  Now observation~\ref{obs:building-skeleta-rep}, truncated at $\sk_{\deg(c)-1}\scat{C}$, tells us that the inclusion $\emptyset\inc\boundary\scat{C}^c$ may be expressed as a transfinite composite of pushouts of boundary maps $\boundary\scat{C}^d\inc\scat{C}^d$ in which $d$ occurs as the domain  of some non-identity map $ d\to c$ in $\direct{\scat{C}}$, which means in particular that $\deg(d) < \deg(c)$. Now applying lemma~\ref{lem:leibniz-tcofp} this in turn implies that the Leibniz tensor $(\emptyset\inc\boundary\scat{C}^c)\leib\wcolim_{\scat{C}}f$ may be expressed as a transfinite composite of pushouts of relative latching maps $\leib{\latch^d}f=(\boundary\scat{C}^d\inc\scat{C}^d)\leib\wcolim_{\scat{C}}f$ for which $\deg(d)<\deg(c)$. However, by the inductive hypothesis we have already shown that for each object $d\in\scat{C}$ with $\deg(d) < \deg(c)$ the relative latching map $\leib{\latch^d}f$ is a trivial cofibration so it follows that $(\emptyset\inc\boundary\scat{C}^c)\leib\wcolim_{\scat{C}}f$ is a transfinite composite of pushouts of trivial cofibrations and is thus itself a trivial cofibration.

  So, in summary, we discover that $f^c$ is a composite of the relative latching map $\leib{\latch^c}f$ and a pushout of the trivial cofibration $(\emptyset\inc\boundary\scat{C}^c)\leib\wcolim_{\scat{C}}f$. Of course, any pushout of a trivial cofibration is a trivial cofibration, and thus a weak equivalence, so by the 2-of-3 property for weak equivalences it follows that the relative latching map $\leib{\latch^c}f$ is a weak equivalence if and only if $f^c$ is a weak equivalence.
  \end{proof}
  
  \begin{obs}[Reedy cofibrations are pointwise cofibrations]
Note that the proof of this result, which expresses $f^c$ as a composite of the relative latching map $\leib{\latch^c}f$ with a transfinite composite of pushouts of relative latching maps, also demonstrates that a Reedy cofibration is a pointwise cofibration. Dually, a Reedy fibration is a pointwise fibration. In particular, a Reedy cofibrant diagram is pointwise cofibrant and a Reedy fibrant diagram is pointwise fibrant. The converse implications do not hold.
  \end{obs}
  
\begin{lem}[lifting]\label{lem:lifting} Suppose $i$ is a Reedy cofibration and $p$ is a Reedy fibration in $\lcat{M}^{\scat{C}}$. If either $i$ or $p$ is a pointwise weak equivalence, then any lifting problem \eqref{eq:i-lift-p} has a solution.
\end{lem}
\begin{proof} This result follows from the construction of observation~\ref{obs:relative.lifting}, but we prefer a different argument. By proposition \ref{prop:building-up}, $i$ can be expressed as a cell complex whose cells have the form $(\boundary\scat{C}_c\inc\scat{C}_c) \leib\tns \leib{\latch^c}i$. Hence, it suffices to show that for any $c \in \scat{C}$, the map $(\boundary\scat{C}_c\inc\scat{C}_c) \leib\tns \leib{\latch^c}i$ lifts against $p$. By observation \ref{obs:leibniz-lifting-properties} a lifting problem of the form displayed on the left transposes to a lifting problem between $\leib{\latch^c} i$ and $\leibwlim{\boundary\scat{C}_c\inc\scat{C}}{p}_\scat{C} \cong \leib{\match^c}p$. By lemma \ref{lem:triv-cof-char} and its dual, the model structure on $\lcat{M}$ provides a solution to this lifting problem.
\end{proof}

\begin{lem}[factorisation]\label{lem:factorisation} Any map $f \colon X \to Y$ in $\lcat{M}^{\scat{C}}$ can be factorised as a Reedy trivial cofibration followed by a Reedy fibration and as a Reedy cofibration followed by a Reedy trivial fibration.
\end{lem}
\begin{proof}
As one might expect, we define these factorisations inductively using observation \ref{obs:inductive.defn}. The factorisations on $\lcat{M}^{\scat{C}}$ may be defined functorially if the corresponding factorisations in the model structure on $\lcat{M}$ are functorial. To begin, we factorise the maps $f^c$ for all objects $c$ with degree zero using the factorisation on $\lcat{M}$.

Suppose now that we have defined the components of an appropriate factorisation $X^d \to Z^d \to Y^d$ of $f^d$ for each object $d$ with $\deg(d) < \deg(c)$. By lemma \ref{lem:inductive.defn} and observation \ref{obs:inductive.defn}, to define the attendant factorisation of $f^c$, it suffices to define an object $Z^c$ of $\lcat{M}$ together with the dotted arrow maps \[ \xymatrix{ \latch^c X \ar[r] \ar[d] & X^c \ar@{-->}[d] \ar[r] & \match^c X \ar[d] \\ \latch^c Z \ar@{-->}[r] \ar[d] & Z^c \ar@{-->}[r] \ar@{-->}[d] & \match^c Z \ar[d] \\ \latch^c Y \ar[r] & Y^c \ar[r] & \match^c Y}\] The object $Z^c$ and the dotted arrows are defined by using the model structure on $\lcat{M}$ to factorise the map \begin{equation}\label{eq:factorisation-defn} \xymatrix{ X^c \cup_{\latch^c X} \latch^c Z \ar@{-->}[r] & Z^c \ar@{-->}[r] &  Y^c \times_{\match^c Y} \match^c Z} \end{equation} defined using the solid arrows. Note that by construction, the left-hand map of \eqref{eq:factorisation-defn} is the relative latching map of $X^c \to Z^c$, while the right-hand map is the relative matching map of $Z^c \to Y^c$. Hence, lemma \ref{lem:triv-cof-char} implies that this construction defines the desired Reedy factorisation.
\end{proof}

With these lemmas, it is straightforward to establish the Reedy model structure.

\begin{proof}[Proof of theorem \ref{thm:model-structure}] 
In the presence of a class of weak equivalences satisfying the 2-of-3 property, a class of cofibrations and a class of fibrations define a model structure if and only if there are a pair of \emph{weak factorisation systems} given by the trivial cofibrations and fibrations and the cofibrations and trivial fibrations \cite{Joyal:2007kk}. Two classes of maps form a weak factorisation system if they satisfy the lifting and factorisation properties of lemmas \ref{lem:lifting} and \ref{lem:factorisation} and if each class is closed under retracts. This final property follows from the functoriality of the constructions of relative latching and matching maps and lemma \ref{lem:triv-cof-char}.
\end{proof}

\begin{rec}[cofibrantly generated model categories]
  A model category $\lcat{M}$ is {\em cofibrantly generated\/} if there exist sets $\mclass{I}$ and $\mclass{J}$ of cofibrations and trivial cofibrations for which the retract closures of $\cell(\mclass{I})$ and $\cell(\mclass{J})$ are the classes of cofibrations and trivial cofibrations of $\lcat{M}$, respectively.
\end{rec}

Given the hard work already undertaken in sections~\ref{sec:Leibniz-rel-cell-cx} and~\ref{sec:cellular}, the following important proposition is now somewhat of a triviality to prove:

\begin{prop}\label{prop:reedy.cof.gen}
  Suppose that $\lcat{M}$ is a cofibrantly generated model category, with sets $\mclass{I}$ and $\mclass{J}$ of generating cofibrations and trivial cofibrations respectively. Then the corresponding Reedy model category $\lcat{M}^{\scat{C}}$ is also cofibrantly generated, with sets $\mclass{B}\leib\tns\mclass{I}$ and $\mclass{B}\leib\tns\mclass{J}$ of generating cofibrations and trivial cofibrations respectively.
\end{prop}

\begin{proof}
  First observe that, since $\mclass{B}$ is a set (it has only as many elements as the small category $\scat{C}$ has objects), it follows therefore that $\mclass{B}\leib\tns\mclass{I}$ and $\mclass{B}\leib\tns\mclass{J}$ are sets as required.

An immediately corollary of \ref{cor:building-up} is that the maps in $\cell(\mclass{B}\leib\tns\mclass{I})$ (respectively in $\cell(\mclass{B}\leib\tns\mclass{J})$) are Reedy (trivial) cofibrations. Conversely, a map $f\colon X\to Y$ of $\lcat{M}^{\scat{C}}$ is a cofibration (respectively a trivial cofibration) in the Reedy model structure if and only if each of its relative latching maps is a retract of a map in $\cell(\mclass{I})$ (respectively in $\cell(\mclass{J})$). Any functor preserves retracts; in particular, it is well known that retracts commute with the formation of cell complexes. It follows, as in the proof of corollary \ref{cor:building-up} that proposition \ref{prop:building-up} implies that $f$ is the retract of a map in $\cell(\mclass{B}\leib\tns\mclass{I})$ (respectively in $\cell(\mclass{B}\leib\tns\mclass{J})$), as claimed.
\end{proof}

\begin{ex} Example \ref{ex:EZ.simp.set} extends to simplicial objects in $\Set$-valued functor categories. In particular, the Eilenberg-Zilber lemma implies that a map of bisimplicial sets is a monomorphism if and only if its relative latching maps are monomorphisms in $\sSet$. Hence, the Reedy model structure coincides with the \emph{injective model structure}, whose weak equivalences and cofibrations are defined pointwise. Proposition \ref{prop:reedy.cof.gen} implies further that this model structure is cofibrantly generated.
\end{ex}

\section{Homotopy limits and colimits}\label{sec:hocolim}

By hypothesis, a model category $\lcat{M}$ necessarily has all limits and colimits. However, it need not be the case that the limits or colimits of pointwise weakly equivalent diagrams are themselves weakly equivalent. Informally, it is common to say that for certain special diagrams, the  limit or colimit somehow has the ``correct'' homotopy type, in which case it is called an \emph{homotopy limit} or \emph{homotopy colimit}. As this terminology suggests,  pointwise weakly equivalent diagrams of this type will have weakly equivalent homotopy limits or colimits.

In good settings there are formulae to compute the homotopy limit or homotopy colimit of any diagram, regardless of whether the ordinary limit or colimit happen to be homotopically correct. The homotopy limit is defined to be a right derived functor of the limit functor, and the homotopy colimit is defined to be a left derived functor of the colimit functor. Here we mean ``point-set level'' derived functors, whose output is an object of $\lcat{M}$ rather than an object of the homotopy category. (As a caveat, this use of ``functor'' should only be interpreted literally in the case where the model category $\lcat{M}$ is supposed to have functorial factorisations; for convenience of language, let us tacitly suppose this is the case henceforth.)

\begin{defn}[homotopy limits and colimits]
The special cases of homotopy limits and colimits considered here are defined via the following definition-schema. Observe that the constant diagram functor $\lcat{M} \to \lcat{M}^\scat{C}$ carries weak equivalences to pointwise weak equivalences. If $\scat{C}$ is a category admitting a Reedy structure in such a way that the constant diagram functor carries cofibrations in $\lcat{M}$ to Reedy cofibrations in $\lcat{M}^\scat{C}$, then the constant diagram functor is left Quillen  with respect to the Reedy model structure. It follows that its right adjoint $\lim \colon \lcat{M}^\scat{C} \to \lcat{M}$ is right Quillen. Hence, by Ken Brown's lemma, the limits of pointwise weakly equivalent Reedy fibrant diagrams are weakly equivalent. These Reedy fibrant diagrams are those diagrams whose limits are understood to be ``homotopically correct''. The homotopy limit functor is defined by replacing a given diagram by a pointwise weakly equivalent Reedy fibrant diagram and then computing the limit. This replacement is computed via a fibrant replacement in the Reedy model structure, which is functorially constructed by lemma \ref{lem:factorisation}.

Dually, when the constant diagram functor carries fibrations in $\lcat{M}$ to Reedy fibrations in $\lcat{M}^\scat{C}$, its left adjoint $\colim \colon \lcat{M}^\scat{C} \to \lcat{M}$ is left Quillen with respect to the Reedy model structure. Hence, colimits of weakly equivalent Reedy cofibrant diagrams are weakly equivalent and understood to be ``homotopically correct''. The homotopy colimit functor is defined to be the colimit of a functorial Reedy cofibrant replacement of the original diagram.
\end{defn}

Let us now implement this outline to deduce formulae for homotopy limits and homotopy colimits of diagrams indexed by particular Reedy categories.

\begin{ex}[homotopy coequalisers]\label{ex:hocoeq}
Give the category $a \rightrightarrows b$ the Reedy structure described in example \ref{ex:parallel.pair.reedy}. As described in example \ref{ex:parallel.pair.latching}, for any diagram $X$ with this shape, the matching objects $\match^aX$ and $\match^bX$ are terminal, from which we deduce that the relative matching maps associated to a natural transformation $X \to Y$ are just the components of that natural transformation. The constant diagram functor is manifestly right Quillen, from which we conclude that the coequaliser of the diagram $X^a \rightrightarrows X^b$ is the homotopy coequaliser if it is Reedy cofibrant: i.e., if $X^a$ is cofibrant and $X^a \coprod X^a \to X^b$ is a cofibration.

Given an arbitrary diagram $X^a \rightrightarrows X^b$, its Reedy cofibrant replacement is  defined by first taking a cofibrant replacement $\overline{X}^a \xrightarrow{\sim} X^a$ and then factoring the natural map \[ \xymatrix{ \overline{X}^a \coprod \overline{X}^a \ar[d]_{\rotatebox{90}{$\sim$}} \ar@{>->}[r] & \overline{X}^b \ar[d]^{\rotatebox{90}{$\sim$}} \\ X^a \coprod X^a \ar[r] & X^b}\] as a cofibration followed by a weak equivalence. The coequaliser of $\overline{X}^a \rightrightarrows \overline{X}^b$ is the homotopy coequaliser of $X^a \rightrightarrows X^b$.
\end{ex}

\begin{ex}[homotopy equalisers]\label{ex:hoeq}
By contrast, the constant diagram functor is unlikely to be \emph{left} Quillen when $a \rightrightarrows b$ is given the Reedy category structure of example \ref{ex:parallel.pair.reedy}. The $b$th relative latching map associated to the image of a cofibration $U \rightarrowtail V$ is \[ \xymatrix{ U \coprod U \ar@{>->}[d] \ar[r]^-\nabla & U \ar[d] \ar@{>->}@/^/[ddr] \\ V \coprod V \ar@/_/[drr]_\nabla \ar[r] & \cdot \poexcursion \ar@{-->}[dr]|{\leib{\latch^b}} \\ & & V}\] where ``$\nabla$'' denotes the fold map. This is unlikely to be a cofibration; for instance, if $U \rightarrowtail V$ is a monomorphism, $\leib{\latch^b}$ need not be a monomorphism.

By contrast if $a \rightrightarrows b$ is given the opposite Reedy category structure, as described in example \ref{ex:parallel.pair.latching}, then the constant diagram functor is left Quillen, and hence we see that the homotopy equaliser of a diagram $X^a \rightrightarrows X^b$ is defined to be the equaliser of its Reedy fibrant replacement $\overline{X}^a \rightrightarrows \overline{X}^b$, constructed from a fibrant replacement $X^b \xrightarrow{\sim} \overline{X}^b$ via the factorisation \[ \xymatrix{ X^a \ar[r] \ar[d]_{\rotatebox{90}{$\sim$}} & X^b \times X^b \ar[d]^{\rotatebox{90}{$\sim$}} \\ \overline{X}^a \ar@{->>}[r] & \overline{X}^b \times \overline{X}^b.}\] 
\end{ex}

The general form of the dualisation just observed is worth recording:

\begin{prop}\label{prop:reedy.model.dual} Let $\scat{C}$ be a Reedy category. If the constant diagram functor $\lcat{M} \to \lcat{M}^\scat{C}$ is right Quillen with respect to the Reedy model structure associated to any model category $\lcat{M}$, then the constant diagram functor $\lcat{M} \to \lcat{M}^{\scat{C}\op}$ is left Quillen with respect to the Reedy model structure defined with respect to the dual Reedy category $\scat{C}\op$.
\end{prop}
\begin{proof}
The proof is an exercise in the application of the principle of duality, left to the reader with the following hints: the passage from a model category $\lcat{M}$ to its opposite exchanges the cofibrations and the fibrations, while the passage from $(\lcat{M}\op)^\scat{C}$ to its opposite $\lcat{M}^{(\scat{C}\op)}$ exchanges relative matching maps with relative latching maps defined with respect to the dual Reedy category structure.
\end{proof}

\begin{ex}[mapping telescopes]\label{ex:mapping}
Suppose $X$ is a sequence of maps \begin{equation}\label{eq:sequence} X^0 \to X^1 \to X^2 \to \cdots \end{equation} in a model category $\lcat{M}$. Assigning the poset $\omega$ the Reedy category structure of example \ref{ex:poset.reedy}, we deduce from example \ref{ex:sequence.latching} that the relative latching maps associated to a natural transformation $X \to Y$ are the components of the natural transformation with shifted index. In particular, the constant diagram functor is right Quillen with respect to the Reedy model structure, from which we deduce that the homotopy colimit of $X$ is computed by the sequential colimit of its Reedy cofibrant replacement. By example \ref{ex:sequence.latching},  \eqref{eq:sequence} is Reedy cofibrant just when it is a sequence of cofibrations between cofibrant objects. By lemma \ref{lem:factorisation}, the Reedy cofibrant replacement is defined inductively by taking a cofibrant replacement of $X^0$, and then replacing each map in turn by a cofibration whose domain is the previously defined cofibrant object:
\begin{equation}\label{eq:sequence.replacement} \xymatrix{ \overline{X}^0 \ar[d]_{\rotatebox{90}{$\sim$}} \ar@{>->}[r] & \overline{X}^1 \ar[d]_{\rotatebox{90}{$\sim$}} \ar@{>->}[r] & \overline{X}^2 \ar[d]_{\rotatebox{90}{$\sim$}} \ar@{>->}[r] & \cdots \\ X^0 \ar[r] & X^1 \ar[r] & X^2 \ar[r] & \cdots}\end{equation} 
In the category of topological spaces, this homotopy colimit is called the mapping telescope.

Dually, proposition \ref{prop:reedy.model.dual} implies that the limit of a diagram \[ \cdots \to X_2 \to X_1 \to X_0 \] is its homotopy limit if $X$ consists of fibrations between fibrant objects. The homotopy limit is defined to be the limit of a Reedy fibrant replacement, constructed dually to \eqref{eq:sequence.replacement}.
\end{ex}

\begin{ex}[homotopy pushouts]\label{ex:hopushout}
Consider the Reedy structures assigned to the category $b \leftarrow a \to c$ in examples \ref{ex:poset.reedy} and \ref{ex:pushout.reedy.alt}. A pushout diagram $X$ is Reedy cofibrant in the first case just when $X^a$, $X^b$, and $X^c$ are cofibrant objects and both maps are cofibrations; in the second case, the diagram must again be pointwise cofibrant, but only $X^a \to X^c$ must be a cofibration.

We claim that the constant diagram functor is right Quillen with respect to this latter Reedy category structure; the result in the other case is easier. Given a natural transformation \begin{equation}\label{eq:pushout.compare} \xymatrix{ X^b \ar[d]_{f^b} & X^a \ar[l] \ar[r] \ar[d]^{f^a} & X^c \ar[d]^{f^c} \\ Y^b & Y^a \ar[l] \ar[r] & Y^c}\end{equation} The relative matching maps are the components $f^a$ and $f^c$ together with the map from $X^a$ to the pullback of $f^b$ along $Y^a \to Y^b$. In the image of the constant diagram functor, the horizontal maps are identities, and $f^a$, $f^b$, and $f^c$ coincide. In particular, the left-hand square of \eqref{eq:pushout.compare} is a pullback square, so the relative matching map at $b$ is an isomorphism. It follows that the constant diagram functor preserves fibrations, as claimed.

The upshot is that the pushout of any diagram whose objects are cofibrant and in which at least one of the maps is a cofibration is a homotopy pushout. The homotopy pushout of a generic diagram $X$ can be formed by replacing $X^a$ by a cofibrant object $\overline{X}^a$ and then factorising the composite maps $X^b \leftarrow \overline{X}^a \to X^c$ as cofibrations followed by weak equivalences. 

Dually, proposition \ref{prop:reedy.model.dual} implies that the pullback of a diagram consisting of at least one fibration between three fibrant objects is a homotopy pullback, and the homotopy pullback of a generic diagram can be computed by replacing the objects by fibrant objects and at least one of the maps by a fibration.
\end{ex}

\begin{ex}\label{ex:stupid-simplicial} 
The constant diagram functor is not right Quillen with respect to the Reedy model structure on category of simplicial objects. This is perhaps unsurprising: only rarely would one expect the ordinary colimit of a simplicial object $X$,  isomorphic to the coequaliser of the two face maps $X_1 \rightrightarrows X_0$, to have the correct homotopy type. By contrast, it is left Quillen because the positive-degree latching maps of a constant simplicial object are isomorphisms. However, the associated homotopy limits are not very interesting: the limit of a simplicial object is computed by evaluating at $[0]$, the initial object in $\Del\op$. Dual remarks of course apply to cosimplicial objects.
\end{ex}

\section{Connected weights}\label{sec:connected-weights}

In this section, we apply the the theory developed in the previous sections of this paper to unify, extend, and clarify the computations just given of homotopy limits and colimits of diagrams indexed by Reedy categories. Our methods are, unsurprisingly, all in the weights. More precisely, we shall see that there is a simple condition on the weights for the latching or matching objects associated to a Reedy category $\scat{C}$ that is necessary and sufficient for  the limit or colimit functors $\lim,\colim \colon \lcat{M}^\scat{C} \to \lcat{M}$ to be, respectively, right and left Quillen for \emph{any} model category $\lcat{M}$. This calculation with the weights illustrates why our ad-hoc arguments about the behavior of the constant diagram functor with respect to cofibrations and fibrations worked for certain Reedy categories but not for others.

  \begin{prop}\label{prop:2/3-SM7}
    Suppose that $\lcat{M}$ is a model category and that $f$ is a map in $\Set^{\scat{C}\op}$ whose relative latching maps are all monomorphisms. If $i$ is a Reedy (trivial) cofibration in $\lcat{M}^{\scat{C}}$ then the Leibniz colimit $f\leib\wcolim_{\scat{C}} i$ is a (trivial) cofibration in $\lcat{M}$. 
    \end{prop}

Recall that the relative latching maps of $\emptyset \to X$ are the latching maps of $X$.      In particular, it follows immediately that if $X$ is an object in $\Set^{\scat{C}\op}$ whose latching maps are all monomorphisms then the the functor $X\wcolim_{\scat{C}}{-}\colon\lcat{M}^{\scat{C}}\to \lcat{M}$ is a left Quillen functor with respect to the model structure on $\lcat{M}$ and the corresponding Reedy model structure on $\lcat{M}^{\scat{C}}$.
  
  \begin{proof}
     By corollary~\ref{cor:B-cell-complex}, we know that $f$ admits a presentation of a cell complex whose cells are boundary inclusions $\boundary\scat{C}^c\inc\scat{C}^c$. So we may apply lemma~\ref{lem:leibniz-tcofp} to show that $f\leib\wcolim_{\scat{C}} i$ admits a presentation as a cell complex whose cells are $(\boundary\scat{C}^c\inc\scat{C}^c)\leib\wcolim_{\scat{C}}i= \leib{\latch^c}i$. However, if $i$ is a Reedy (trivial) cofibration in $\lcat{M}^{\scat{C}}$ then, each of its relative latching maps $\leib{\latch^c}i$ is a (trivial) cofibration in $\lcat{M}$. So in that case we have succeeded in showing that $f\leib\wcolim_{\scat{C}} i$ admits a presentation as a cell complex whose cells are (trivial) cofibrations and consequently it too is a (trivial) cofibration in $\lcat{M}$ as required.
  \end{proof}
  
 \begin{obs}\label{obs:fibrant.constants}
Consider $1 \in \Set^{\scat{C}\op}$, the constant diagram at the terminal object. If the latching maps of $1$ are monomorphisms, then proposition \ref{prop:2/3-SM7} implies that $1 \wcolim_{\scat{C}} - \colon \lcat{M}^\scat{C} \to \lcat{M}$ is a left Quillen functor. But in example \ref{ex:weighted-conical}, we saw that $1\wcolim_{\scat{C}}-$ is exactly the colimit functor! The dual to proposition \ref{prop:2/3-SM7}, obtained by replacing the model category $\lcat{M}$ with its opposite (and then also $\scat{C}$ with its opposite, for aesthetic reasons), says that if the latching maps of $1 \in \Set^{\scat{C}}$ are monomorphisms, then $ \wlim{1}{-}_{\scat{C}} \cong \lim \colon \lcat{M}^\scat{C} \to \lcat{M}$ is a right Quillen functor.
\end{obs}

To apply observation \ref{obs:fibrant.constants}, we must describe conditions on the Reedy category $\scat{C}$ so that the constant $\scat{C}$-diagram $1$ has monomorphic latching maps. By observation \ref{obs:latching.ordinary.colimit}, the latching object at $c \in \scat{C}$ of the constant diagram at $1$ is  the colimit of the constant diagram at $1$ indexed by the category of elements for the weight $\boundary\scat{C}^c$. The colimit of a constant diagram is the coproduct of the single object indexed over the set of connected components. In particular, the latching map $l^c$, whose codomain is $1$, is a monomorphism if and only if for each $c \in \scat{C}$, the category $\el\boundary\scat{C}^c$ is either empty or connected so that this coproduct is either $\emptyset$ or $1$.

\begin{defn}[connected weights] Say a weight $W \in \Set^\scat{C}$ is \emph{connected} if it is empty or if either of the following equivalent conditions are satisfied: \begin{enumerate} \item The category $\el W$ is connected.
\item The functor $W$ cannot be expressed as a coproduct $W \cong W' \coprod W''$ with both $W'$ and $W''$ non-empty.
\end{enumerate}
\end{defn}

Combining observation \ref{obs:fibrant.constants} with the terminology just introduced, we have the following corollary of proposition \ref{prop:2/3-SM7}.

\begin{cor}\label{cor:connected.weights} If $\scat{C}$ is a Reedy category so that each $\boundary\scat{C}^c$ is connected, then for any model category $\lcat{M}$, $\lim \colon \lcat{M}^\scat{C} \to \lcat{M}$ is a right Quillen functor. Dually, if instead each $\boundary\scat{C}_c$ is connected, then $\colim \colon \lcat{M}^\scat{C} \to \lcat{M}$ is a left Quillen functor.
\end{cor}

We like our statement of corollary \ref{cor:connected.weights} because it makes it clear that ``it is all in the weights''. For the reader's convenience, we note that this condition is expressed in another way in the standard literature.
  
  \begin{defn}[cofibrant constants] A Reedy category $\scat{C}$ has \emph{cofibrant constants} if the constant $\scat{C}$-diagram at any cofibrant object in any model category is Reedy cofibrant. Dually, $\scat{C}$ has \emph{fibrant constants} if the constant $\scat{C}$-diagram at any fibrant object in any model category is Reedy fibrant.
   \end{defn}

\begin{lem} The weights $\boundary\scat{C}^c$ are all connected if and only if $\scat{C}$ has cofibrant constants.
\end{lem}
\begin{proof}
To obtain the ``only if'' direction, recall that the cofibrations of the canonical model structure on the category of simplicial sets are the monomorphisms. So in particular, we know that all simplicial sets, including the $0$-simplex $\Del^0$, are cofibrant in there. So if $\scat{C}$ has cofibrant constants then the constant $\scat{C}$-diagram on $\Del^0$ is Reedy cofibrant, that is to say all of its latching maps are monomorphisms. Applying the functor $({-})_0\colon \sSet\to\Set$, which carries each simplicial set to its set of $0$-simplices and preserves all colimits, it follows that the constant $\scat{C}$-diagram on the terminal set $1 = (\Del^0)_0$ has latching maps which are all monomorphisms. As argued after observation~\ref{obs:fibrant.constants} that in turn implies that each weight $\boundary\scat{C}^c$ is connected, as required.

Conversely, as a consequence of observation \ref{obs:latching.ordinary.colimit}, for any cofibrant object $M$ in any model category $\lcat{M}$, because the weights for the latching objects are connected, the latching objects of the constant diagram at $M$ are either $\emptyset$ or $M$, and the latching maps are either $\emptyset \to M$ or the identity at $M$. Hence, it follows that such diagrams are Reedy cofibrant, which means that $\scat{C}$ has cofibrant constants.
\end{proof}

\begin{ex} The Reedy categories indexing countable sequences, pushout diagrams, coequaliser diagrams, and cosimplicial objects all have fibrant constants; many of the weights for matching objects are empty. Corollary \ref{cor:connected.weights} implies that the ``composition'', pushout, coequaliser, and ``evaluate at [0]'' functors are left Quillen with respect to the Reedy model structures. Dually, the opposite Reedy categories have cofibrant constants, implying that the inverse limit, pullback, equaliser, and ``evaluate at [0]'' functors, respectively, are right Quillen with respect to the Reedy model structures.  This formal calculation in the weights unifies and extends the conclusions of examples \ref{ex:hocoeq}, \ref{ex:hoeq}, \ref{ex:mapping}, \ref{ex:hopushout}, and \ref{ex:stupid-simplicial}.
\end{ex}

\section{Simplicial model categories and geometric realization}\label{sec:simplicial}

Suppose now that $\lcat{M}$, in addition to being complete and cocomplete, is also tensored, cotensored, and enriched over the category of simplicial sets. The tensor and cotensor are defined to be adjoints to the hom-space bifunctor $\hom \colon \lcat{M}\op \times \lcat{M} \to \sSet$ so that the adjunction is encoded by natural isomorphisms \[ \hom (K \tns  M,N) \cong \hom(M, K \pwr N) \cong \hom(K, \hom(M,N)) \qquad \forall K \in \sSet,\ M,N \in \lcat{M} \] of hom-spaces, not simply of hom-sets. It follows that the three bifunctors are simplicially enriched. While we have overloaded the notation ``$\tns$'' and ``$\pwr$,'' there is no ambiguity: the tensor or cotensor with a set is always isomorphic to the tensor or cotensor, respectively, with the corresponding discrete simplicial set.

\begin{defn} Suppose $\lcat{M}$ is a model category that is tensored, cotensored, and enriched over simplicial sets. Then $\lcat{M}$ is a \emph{simplicial model category} if it additionally satisfies the ``SM7'' axiom: \begin{enumerate}[label=(SM7\roman*)] \item The Leibniz tensor sends a monomorphism of simplicial sets and a cofibration in $\lcat{M}$ to a cofibration in $\lcat{M}$.
\item The Leibniz tensor sends a monomorphism of simplicial sets and a trivial cofibration in $\lcat{M}$ to a trivial cofibration in $\lcat{M}$.
\item The Leibniz tensor sends an anodyne map of simplicial sets and a cofibration in $\lcat{M}$ to a trivial cofibration in $\lcat{M}$.
\end{enumerate}
\end{defn}
By observation \ref{obs:leibniz-lifting-properties}, the SM7 axiom has dual forms expressed using the Leibniz bifunctors associated to $\pwr$ or $\hom$. 

\begin{ex} Quillen's original model structure on simplicial sets is a simplicial model category \cite{Quillen:1967:Model}.
\end{ex}

The axioms (SM7i-iii) assert that $\tns$ is a \emph{left Quillen bifunctor} or dually that $\pwr$ and $\hom$ are \emph{right Quillen bifunctors} with respect to the given model structure on $\lcat{M}$ and the Quillen model structure on $\sSet$. More generally, a bifunctor between model categories is a left Quillen bifunctor if its Leibniz bifunctor carries a pair of cofibrations to a cofibration that is acyclic if either of the domain cofibrations is.
	
   \begin{thm}\label{thm:simp.model.cat} Let $\scat{C}$ be a Reedy category and let $\lcat{M}$ be a simplicial model category. Then the weighted colimit and weighted limit bifunctors \[ \sSet^{\scat{C}\op} \times \lcat{M}^{\scat{C}} \xrightarrow{\wcolim_\scat{C}} \lcat{M} \qquad\qquad  (\sSet^{\scat{C}})\op \times \lcat{M}^{\scat{C}} \xrightarrow{ \wlim{\,}{\,}_{\scat{C}}} \lcat{M}\] defined as in \eqref{eq:wlimformula} are respectively left and right Quillen bifunctors.
  \end{thm}
  \begin{proof}
  The weighted colimit is adjoint to the bifunctor \[ (\sSet^{\scat{C}\op})\op \times \lcat{M} \xrightarrow{-\pwr-} \lcat{M}^\scat{C}\] 
  built from the simplicial cotensor $\pwr \colon \sSet\op \times \lcat{M} \to \lcat{M}$; given $A \in \sSet^{\scat{C}\op}$ and $X \in \lcat{M}$ define $(A \pwr X)^c = A_c \pwr X$. Observe that the simplicial cotensor sends weighted colimits in its first variable to weighted limits.
  
   Suppose $i \colon A \to B$ is in $\sSet^{\scat{C}\op}$ and suppose $f \colon X \to Y$ is in $\lcat{M}$. The relative matching map of $i \leib\pwr f \colon B \pwr X \to (B \pwr Y) \times_{A \pwr Y} (A \pwr X)$ is \begin{equation}\label{eq:simpmodelreedy} \leib{\match^c} (i \leib\pwr f) \cong \leibwlim{\boundary\scat{C}_c\inc\scat{C}_c}{i \leib\pwr f}_{\scat{C}} \cong  ( {\boundary\scat{C}_c \inc\scat{C}_c}\leibwcolim_{\scat{C}}{i}) \leib\pwr f \cong (\leib{\latch_c} i) \leib\pwr f.\end{equation} Because $\lcat{M}$ is a simplicial model category, $\pwr \colon \sSet\op \times \lcat{M} \to \lcat{M}$ is a right Quillen bifunctor. If $i$ is a Reedy cofibration, then each $\leib{\latch_c} i$ is a cofibration in $\lcat{M}$. If $f$ is a fibration, it follows that \eqref{eq:simpmodelreedy} is a fibration, and hence that $i \leib\pwr f$ is a Reedy fibration. The same argument combined with lemma \ref{lem:triv-cof-char} implies that if either $i$ or $f$ is acyclic, then $i \leib\pwr f$ is too. 
     \end{proof}
     
     Considering the degenerate case of the Leibniz construction \ref{defn:leibniz} when the domain of one of the morphisms is the initial object, one sees that a left Quillen bifunctor becomes an ordinary left Quillen functor when the value of one of the variables is fixed at a cofibrant object.    Hence, an immediate corollary of theorem \ref{thm:simp.model.cat} is that the weighted colimit functor and the weighted limit functor are, respectively, left and right Quillen, provided that the weight is Reedy cofibrant. 
     
     \begin{ex}[the Yoneda embedding is a Reedy cofibrant weight]
     As in the introduction, let $\Delta \colon \Del \to \sSet$ denote the Yoneda embedding. We must show that each latching map $L^n\Delta \to \Delta^n$ is a cofibration. By definition, $\latch^n\Delta = \boundary\Delta^n \wcolim_{\Del} \Delta \cong \boundary\Delta^n$, by the Yoneda lemma. The proof is completed by the familiar observation that the inclusions $\boundary\Delta^n \to \Delta^n$ are among the monomorphisms, the cofibrations in the Quillen model structure.
   \end{ex}
  
 \begin{cor}[homotopy invariance of geometric realization] Geometric realization preserves pointwise weak equivalences between Reedy cofibrant simplicial objects taking values in a simplicial model category. Dually, totalization preserves pointwise weak equivalences between Reedy fibrant cosimplicial objects taking values in a simplicial model category.
 \end{cor}
 \begin{proof}
The geometric realization of a simplicial object in a tensored simplicial category is defined to be the colimit weighted by the Yoneda embedding. Dually, the totalization of a cosimplicial object in a cotensored simplicial category is defined to be the limit weighted by the Yoneda embedding. Any left or right Quillen functor preserves weak equivalences between cofibrant or fibrant objects, respectively, by Ken Brown's lemma.
\end{proof}

\begin{obs}[skeletal filtration of geometric realization]\label{obs:geo-filt}
For any simplicial object $X$ valued in a cocomplete category $\lcat{M}$, by proposition \ref{prop:building-up} there is a cell complex presentation 
\begin{equation}\label{eq:reedy's.po2} \xymatrix{ & &  \Delta^n \tns L_n X \cup \partial\Delta^n \tns X_n \ar[d] \ar[r] & \Delta^n \tns X_n \ar[d] & \\
 \emptyset \ar[r] & \cdots \ar[r] & \sk_{n-1} X \ar[r] & \sk_n X \poexcursion \ar[r] & \cdots \ar[r] & X}\end{equation}
in $\lcat{M}^{\Del\op}$ defined by taking the (unenriched) weighted colimit weighted by the hom bifunctor $\Delta \in \Set^{\Del{\op} \times \Del}$. Here $\tns$ denotes the copower, defined pointwise, of a simplicial set with an object of $\lcat{M}$, as in  example \ref{ex:tensor-cotensor}. This is the presentation described in \eqref{eq:reedy's.po}.

When $\lcat{M}$ is a simplicial model category, we can form the (enriched) weighted colimit of the diagrams displayed in \eqref{eq:reedy's.po2} weighted by the Yoneda embedding $\Delta \in \sSet^{\Del}$. By cocontinuity of the weighted colimit bifunctor and the coYoneda lemma, if $Y$ is a simplicial set and $M \in \lcat{M}$ we have an isomorphism \[ \Delta \wcolim_{\Del\op} ( Y \ast M) \cong Y \tns M,\] in which the $\tns$ appearing on the right-hand side is the simplicial tensor of $\lcat{M}$. Thus, taking the geometric realization of the simplicial objects of \eqref{eq:reedy's.po2}, we obtain the following cell complex presentation in $\lcat{M}$:  
\[ \xymatrix{ & &  \Delta^n \tns L_n X \cup \partial\Delta^n \tns X_n \ar[d] \ar[r] & \Delta^n \tns X_n \ar[d] & \\
 \emptyset \ar[r] & \cdots \ar[r] & |\sk_{n-1} X| \ar[r] & |\sk_n X| \poexcursion \ar[r] & \cdots \ar[r] & |X|}
 \] Using this presentation and the ``SM7'' axiom, it is possible to give an alternate proof of the homotopy invariance of the geometric realization based on the homotopy invariance of pushouts and sequential colimits of cofibrations in the model category $\lcat{M}$.
 
 Dual ``Postnikov tower'' presentations exist for the totalization of a cosimplicial object valued in a simplicial model category.
\end{obs}

\bibliographystyle{amsplain}
\bibliography{index}

\end{document}

%% file: macros.tex
\newcommand{\ifundef}[1]{\expandafter\ifx\csname#1\endcsname\relax}

\DeclareMathAlphabet{\mathbbe}{U}{bbold}{m}{n}

\makeatletter

\def\re@DeclareMathSymbol#1#2#3#4{%
    \let#1=\undefined
    \DeclareMathSymbol{#1}{#2}{#3}{#4}}

\ifundef{Top}
  \DeclareSymbolFont{tcSyC}{U}{txsyc}{m}{n}
  \SetSymbolFont{tcSyC}{bold}{U}{txsyc}{bx}{n}
  \DeclareFontSubstitution{U}{txsyc}{m}{n}

  \re@DeclareMathSymbol{\Top}{\mathord}{tcSyC}{120}
  \re@DeclareMathSymbol{\Bot}{\mathord}{tcSyC}{121}
\fi

\ifundef{righthalfcup}
  \DeclareFontFamily{U}{MnSymbolC}{}
  \DeclareSymbolFont{mnSyC}{U}{MnSymbolC}{m}{n}
  \SetSymbolFont{mnSyC}{bold}{U}{MnSymbolC}{b}{n}
  \DeclareFontShape{U}{MnSymbolC}{m}{n}{
      <-6>  MnSymbolC5
     <6-7>  MnSymbolC6
     <7-8>  MnSymbolC7
     <8-9>  MnSymbolC8
     <9-10> MnSymbolC9
    <10-12> MnSymbolC10
    <12->   MnSymbolC12}{}
  \DeclareFontShape{U}{MnSymbolC}{b}{n}{
      <-6>  MnSymbolC-Bold5
     <6-7>  MnSymbolC-Bold6
     <7-8>  MnSymbolC-Bold7
     <8-9>  MnSymbolC-Bold8
     <9-10> MnSymbolC-Bold9
    <10-12> MnSymbolC-Bold10
    <12->   MnSymbolC-Bold12}{}
  
  \re@DeclareMathSymbol{\righthalfcup}{\mathord}{mnSyC}{184}
  \re@DeclareMathSymbol{\lefthalfcap}{\mathord}{mnSyC}{185}
\fi

\makeatother

\newcommand{\mlaux}[3]{\setbox0=\hbox{$\mathsurround=0pt #2{#3}$}%
  \dimen0=\dp0\advance\dimen0 by \ht0\lower#1\dimen0\box0}

\newcommand{\makellapm}[2]{\hbox to 0pt{\hss$\mathsurround=0pt #1{#2}$}}

\newcommand{\makerlapm}[2]{\hbox to 0pt{$\mathsurround=0pt #1{#2}$\hss}}

\newcommand{\makelapm}[2]{\hbox to 0pt{\hss$\mathsurround=0pt #1{#2}$\hss}}

\newcommand{\makeushort}[3]{%
	\setbox0=\hbox{$\mathsurround=0pt #2{#3}$}%
	\hbox to 1\wd0{\hss\underbar{\hbox to #1\wd0{\hss\box0\hss}}\hss}}
\newcommand{\ushort}[1]{\relax\mathpalette{\makeushort{#1}}}

\def\makebigger#1#2#3{\scalebox{#1}{$\mathsurround=0pt #2{#3}$}}
\def\bigger#1#2{{\relax\mathpalette{\makebigger{#1}}{#2}}}

\def\scaleuphalf{1.0954}

\newcommand{\op}{^{\mathord{\text{\rm op}}}}

\newcommand{\mapcat}{^\cattwo}

\newcommand{\nd}{^{\text{nd}}}

\newcommand{\defeq}{\mathrel{:=}}

\makeatletter

\def\newmop{\@ifstar{\@newmop m}{\@newmop o}}
\def\@newmop#1{\@ifnextchar[{\@@newmop #1}{\@@@newmop #1}}
\def\@@newmop#1[#2]{\@declmathop #1#2}
\def\@@@newmop#1#2{\expandafter\@declmathop\expandafter #1\csname #2\endcsname{#2}}

\makeatother

\newdir{ |}{{}*!/-5pt/@{|}}

\newmop{im}
\newmop{coim}
\newmop{dom}
\newmop{cod}
\newmop{id}

\newmop{obj}
\newmop{arr}
\newmop{sq}
\newmop{norm}

\newmop{el}

\newmop{ev}

\newmop{Ext}
\newmop{icon}

\newmop{cell}
\newmop{cof}
\newmop{fib}

\newmop{diag}

\newmop*{colim}
\newmop*{holim}

\newmop[\lan]{lan}
\newmop[\ran]{ran}

\newcommand{\unit}{\eta}
\newcommand{\counit}{\varepsilon}

\makeatletter
\newcommand{\rotatemath}[2]{\rotatebox[origin=c]{180}{$\m@th #1{#2}$}}
\makeatother

\newcommand{\cprod}{\mathbin{\mathpalette\rotatemath\Pi}}

\newmop[\pr]{pr}
\newmop[\dm]{d} 

\newmop[\cpr]{in}

\newcommand{\pocorner}{\hbox to 8pt{{\vrule height8pt depth0pt width0.5pt}%
    \vbox to 8pt{{\hrule height0.5pt width7.5pt depth0pt}\vfill}}}
\newcommand{\poexcursion}{\save[]-<15pt,-15pt>*{\pocorner}\restore}
\newcommand{\pbcorner}{\vbox to 0pt{\kern 4pt\hbox to 0pt{\kern 4pt%
      \vbox{{\hrule height0.5pt width7.5pt depth0pt}}%
      {\vrule height8pt depth0pt width0.5pt}\hss}\vss}}
\newcommand{\pbexcursion}{\save[]+<5pt,-5pt>*{\pbcorner}\restore}

\newcommand{\pwr}{\pitchfork}
\newcommand{\tns}{\ast}
\newcommand{\wcolim}{\circledast}
\newcommand{\leibwcolim}{\leib\wcolim}
\newcommand{\wlim}[2]{\{#1,#2\}}
\newcommand{\leibwlim}[2]{\{#1,#2\}^{\wedge}}

\newmop{cls}

\newcommand{\leib}[1]{\mathbin{\widehat{#1}}}
\newcommand{\wleib}{\widehat}

\newcommand{\etimes}{\mathbin{\ushort{0.5}{\mathord\times}}}

\newcommand{\category}[1]{\underline{\smash[b]{\text{\rm{#1}}}}}

\newcommand{\lcat}{\mathcal}
\newcommand{\scat}{\mathbf}

\newcommand{\cattwo}{{\bigger{1.12}{\mathbbe{2}}}}
\newcommand{\catone}{{\bigger{1.16}{\mathbbe{1}}}}

\makeatletter

\def\Del@Sym{{\bigger\scaleuphalf{\mathbbe{\Delta}}}}

\def\del@fn{\futurelet\del@next}
\def\del@dn{\def\del@next}

\def\parsedel@{%
  \ifx +\del@next \del@dn+{\Del@Sym_{\mathord{+}}}%
  \else \del@dn {\del@fn\parsedel@@}%
  \fi\del@next}

\def\parsedel@@{%
  \ifx\space@\del@next \expandafter\del@dn\space{\del@fn\parsedel@@}%
  \else\ifx [\del@next \del@dn[{\del@fn\parsedel@@@}%
  \else\ifx _\del@next \del@dn{\Delta}%
  \else\ifx ^\del@next \del@dn{\Delta}%
  \else \del@dn{\Del@Sym}%
  \fi\fi\fi\fi\del@next}

\def\parsedel@@@{%
  \ifx\space@\del@next \expandafter\del@dn\space{\del@fn\parsedel@@@}%
  \else\ifx t\del@next \del@dn t{\Del@Sym_\infty\del@fn\parsedel@@@@}%
  \else\ifx b\del@next \del@dn b{\Del@Sym_{-\infty}\del@fn\parsedel@@@@}%
  \else \del@dn{\errmessage{unexpected modifier}}%
  \fi\fi\fi\del@next}

\def\parsedel@@@@{%
  \ifx\space@\del@next \expandafter\del@dn\space{\del@fn\parsedel@@@@}%
  \else\ifx ]\del@next \del@dn]{}%
  \else \del@dn{\errmessage{expecting close of option block}}%
  \fi\fi\del@next}

\def\Del{\del@fn\parsedel@}

\makeatother

\newcommand{\Set}{\category{Set}}

\newcommand{\fact}{\category{fact}}

\newcommand{\sSet}{\category{sSet}}

\newmop{dec}
\newmop{fatdec}
\newmop{slc}
\newmop{fatslc}

\newmop{ir} 
\newmop{incl}

\newcommand{\boundary}{\partial}

\newcommand{\direct}{\overrightarrow}
\newcommand{\inverse}{\overleftarrow}
\newcommand{\reedycat}[1]{(\scat{#1},\direct{\scat{#1}},\inverse{\scat{#1}})}

\newcommand{\latch}{L}
\newcommand{\match}{M}

\def\reedyfilt#1_#2{#1_{\leq #2}}
\newmop{sk}
\newmop{cosk}
\newmop{res}

\newcommand{\mclass}{\mathcal}

\newmop{map}
\newmop{Ho}

\newmop[\pth]{path}
\newmop{cyl}

\newmop[\const]{c}
  
\newmop{coll}
\newmop{wgt}

\newdir{ >}{{}*!/-7pt/@{>}}
\newdir{u(}{{}*!/-4pt/@^{(}}
\newdir{d(}{{}*!/-4pt/@_{(}}
\newdir{|>}{%
  !/4.5pt/@{|}*:(1,-.2)@^{>}*:(1,+.2)@_{>}*+@{}}

\makeatletter
\def\makeslashed#1#2#3#4#5{#1{\mathpalette{\sla@{#2}{#3}{#4}}{#5}}}

\def\@mathlower#1#2#3{\setbox0=\hbox{$\m@th#2#3$}\lower#1\ht0\box0}
\def\mathlower#1#2{\mathpalette{\@mathlower{#1}}{#2}}
\makeatother

\newcommand{\inc}{\hookrightarrow}

\makeatletter

\def\tens@fn{\futurelet\tens@next}
\def\tens@dn{\def\tens@nextcont}
\newtoks\tens@toks
\def\addtotens@toks#1{\tens@toks=\expandafter{\the\tens@toks#1}}

\def\parsetens@@{%
    \ifx\space@\tens@next \expandafter\tens@dn\space{\tens@fn\parsetens@@}%
    \else\ifx ^\tens@next \tens@dn ^##1{\parsetens@procsep^\addtotens@toks{##1}%
      \tens@fn\parsetens@@}%
    \else\ifx _\tens@next \tens@dn _##1{\parsetens@procsep_\addtotens@toks{##1}%
      \tens@fn\parsetens@@}%
    \else\tens@dn{\ifx *\tens@last \else\addtotens@toks\egroup\fi\the\tens@toks}%
    \fi\fi\fi\tens@nextcont}

\def\parsetens@procsep#1{%
  \ifx *\tens@last \addtotens@toks{#1}\addtotens@toks\bgroup%
  \else\ifx \tens@last\tens@next \addtotens@toks,%
  \else \addtotens@toks\egroup\addtotens@toks\bgroup%
    \addtotens@toks\egroup\addtotens@toks{#1}\addtotens@toks\bgroup%
  \fi\fi\let\tens@last\tens@next}

\newcommand{\tn}[1]{\let\tens@last=*\tens@toks={#1}\tens@fn\parsetens@@}

\makeatother

\def\adjdisplay#1-|#2:#3->#4.{{%
    \xymatrix@R=0em@!C=2.5em{%
      *+[l]{#3} \ar@/_0.55pc/[rr]_-{#2} & {\bot} &
      *+[r]{#4}\ar@/_0.55pc/[ll]_-{#1}}}}

\def\adjdisplaytwo#1-|#2:#3->#4.{{%
\xymatrix@=1.2em{
      {#3}\ar@/_1.5ex/[rr]_-{#2}^-{}="one"
      & & {#4}
      \ar@/_1.5ex/[ll]_-{#1}^-{}="two" 
      \ar@{}"one";"two"|{\bot}
    }}}

\def\adjinline#1-|#2:#3->#4.{{#1}\dashv{#2}:#3\to #4}

\newcommand{\pent}[1]{
  \xybox{
    \POS (0,-15)*+{\a}="0", 
         (-14,-5)*+{\b}="1", 
         (-9,12)*+{\c}="2", 
         (9,12)*+{\d}="3", 
         (14,-5)*+{\e}="4"
    \POS"0" \ar "1"^{\labelstyle \ab}|{}="01"
    \POS"1" \ar "2"^{\labelstyle \bc}|{}="12"
    \POS"2" \ar "3"^{\labelstyle \cd}|{}="23"
    \POS"3" \ar "4"^{\labelstyle \de}|{}="34"
    \POS"0" \ar "4"_{\labelstyle \ae}|{}="04"
    \ifcase #1
    \POS"0" \ar "2"|{\labelstyle \ac}="02"
    \POS"0" \ar "3"|{\labelstyle \ad}="03"
    \POS"02";"1"**{}, ?(0.3) \ar@{=>} ?(0.7)^{\labelstyle \abc}
    \POS"03";"2"**{}, ?(0.25) \ar@{=>} ?(0.5)_{\labelstyle \acd}
    \POS"04";"3"**{}, ?(0.2) \ar@{=>} ?(0.4)_{\labelstyle \ade}
    \or
    \POS"1" \ar "3"|{\labelstyle \bd}="13"
    \POS"1" \ar "4"|{\labelstyle \be}="14"
    \POS"13";"2"**{}, ?(0.3) \ar@{=>} ?(0.7)_{\labelstyle \bcd}
    \POS"14";"3"**{}, ?(0.25) \ar@{=>} ?(0.5)_{\labelstyle \bde}
    \POS"04";"1"**{}, ?(0.25) \ar@{=>} ?(0.5)_{\labelstyle \abe}
    \or
    \POS"2" \ar "4"|{\labelstyle \ce}="24"
    \POS"0" \ar "2"|{\labelstyle \ac}="02"
    \POS"02";"1"**{}, ?(0.3) \ar@{=>} ?(0.7)^{\labelstyle \abc}
    \POS"04";"2"**{}, ?(0.2) \ar@{=>} ?(0.35)_{\labelstyle \ace}
    \POS"24";"3"**{}, ?(0.2) \ar@{=>} ?(0.6)^{\labelstyle \cde}
    \or
    \POS"1" \ar "3"|{\labelstyle \bd}="13"
    \POS"0" \ar "3"|{\labelstyle \ad}="03"
    \POS"04";"3"**{}, ?(0.2) \ar@{=>} ?(0.4)_{\labelstyle \ade}
    \POS"13";"2"**{}, ?(0.3) \ar@{=>} ?(0.7)_{\labelstyle \bcd}
    \POS"03";"1"**{}, ?(0.25) \ar@{=>} ?(0.5)^{\labelstyle \abd}
    \or
    \POS"2" \ar "4"|{\labelstyle \ce}="24"
    \POS"1" \ar "4"|{\labelstyle \be}="14"
    \POS"24";"3"**{}, ?(0.2) \ar@{=>} ?(0.6)^{\labelstyle \cde}
    \POS"04";"1"**{}, ?(0.25) \ar@{=>} ?(0.5)_{\labelstyle \abe}
    \POS"14";"2"**{}, ?(0.25) \ar@{=>} ?(0.5)^{\labelstyle \bce}
    \else\fi
  }
}

\newcommand{\pentofpent}[1]{
  \def\baselen{#1}
  \begin{xy}
    0;<\baselen,0mm>:
    *{\xybox{
        \POS(0,-4)*[o]{\pent 0}="zero"
        \POS(16,40)*[o]{\pent 3}="three"
        \POS(72,40)*[o]{\pent 1}="one"
        \POS(88,-4)*[o]{\pent 4}="four"
        \POS(44,-36)*[o]{\pent 2}="two"
        \ar@<1ex>"zero";"three"^-{\objectstyle\abcd}
        \ar@<1ex>"three";"one"^-{\objectstyle\abde}
        \ar@<1ex>"one";"four"^-{\objectstyle\bcde}
        \ar@<-1ex>"zero";"two"_-{\objectstyle\acde}
        \ar@<-1ex>"two";"four"_-{\objectstyle\abce}
        \ar@{=>}(44,-5);(44,+15)^{\objectstyle\abcde}
     }}
  \end{xy}
}
